\documentclass[11pt,a4paper]{amsart}
\usepackage[lmargin=0.9in,rmargin=0.9in]{geometry}
\usepackage{graphicx}
\usepackage[T1]{fontenc}
\usepackage[utf8]{inputenc}
\usepackage{mathtools}
\usepackage{amsthm}
\usepackage{lmodern}
\usepackage[english, french]{babel}
\usepackage{amssymb}
\usepackage{amsmath,calligra,mathrsfs}
\usepackage{enumitem}
\usepackage{tikz-cd}
\usepackage[mathscr]{euscript}
\usetikzlibrary{decorations.markings}
\usepackage{hyperref}
\usepackage[noabbrev,nameinlink]{cleveref}

\theoremstyle{plain}
\newtheorem{thm}{Theorem}[section]
\newtheorem*{thm*}{Theorem}
\newtheorem{lem}[thm]{Lemma}
\newtheorem{prop}[thm]{Proposition}
\newtheorem{cor}[thm]{Corollary}

\theoremstyle{definition}
\newtheorem{defn}[thm]{Definition}
\newtheorem{nota}[thm]{Notation}
\newtheorem{constr}[thm]{Construction}
\newtheorem{ex}[thm]{Example}

\theoremstyle{remark}
\newtheorem{rem}[thm]{Remark}

\Crefname{thm}{Theorem}{Theorems}
\Crefname{lem}{Lemma}{Lemmata}
\Crefname{prop}{Proposition}{Propositions}
\Crefname{cor}{Corollary}{Corollaries}
\Crefname{conj}{Conjecture}{Conjectures}
\Crefname{defn}{Definition}{Definitions}
\Crefname{nota}{Notation}{Notations}
\Crefname{constr}{Construction}{Constructions}
\Crefname{ex}{Example}{Examples}
\Crefname{hyp}{Hypothesis}{Hypotheses}
\Crefname{rem}{Remark}{Remarks}

\newcommand{\Z}{\mathbb{Z}}
\newcommand{\Q}{\mathbb{Q}}

\newcommand{\C}{\mathbb{C}}

\newcommand{\AbHull}{\mathfrak{A}}
\newcommand{\vect}{\mathsf{vect}}

\newcommand{\Spec}{\textup{Spec}}

\newcommand{\MNori}{\mathcal{M}}
\newcommand{\MNoriLoc}{\mathcal{ML}}

\newcommand{\MHS}{\mathsf{MHS}}
\newcommand{\VHS}{\mathsf{VHS}}

\newcommand{\Gmot}{\mathcal{G}_{\textup{mot}}}
\newcommand{\Gal}{\textup{Gal}}

\newcommand{\et}{\textup{{\'e}t}}
\newcommand{\alg}{\textup{alg}}

\newcommand{\Artin}{\textup{Art}}
\newcommand{\Hodge}{\textup{Hdg}}
\newcommand{\pure}{\textup{pure}}

\newcommand{\DAct}{\mathbf{DA}_{\textup{ct}}}
\newcommand{\Bti}{\mathsf{Bti}}
\newcommand{\Real}{\mathsf{R}}
\newcommand{\unit}{\mathbf{1}}

\newcommand{\Perv}{\mathsf{Perv}}
\newcommand{\Loc}{\mathsf{Loc}}

\newcommand{\MotLocus}{\mathsf{EL}}
\newcommand{\WeightLocus}{\mathsf{WL}}
\newcommand{\SplitLocus}{\mathsf{SL}}

\newcommand{\twocolim}{\textup{2-colim}}

\newcommand{\Ical}{\mathcal{I}}

\newcommand{\Tannakian}{\mathcal{T}}

\newcommand{\Sfrak}{\mathfrak{S}}

\newcommand\blfootnote[1]{%
	\begingroup
	\renewcommand\thefootnote{}\footnote{#1}%
	\addtocounter{footnote}{-1}%
	\endgroup
}

\title[The exceptional locus of a motivic local system]{The exceptional locus of a motivic local system}
\author{Luca Terenzi}
\address{Fakult{\"a}t f{\"u}r Mathematik\newline
\indent Universit{\"a}t Regensburg\newline
\indent 93040 Regensburg (Germany)}
\email{luca.terenzi@mathematik.uni-regensburg.de}

\makeatletter
\hypersetup{
	pdfauthor={\author},
	pdftitle={\@title},
	colorlinks,
	linkcolor=[rgb]{0.2,0.2,0.6},
	citecolor=[rgb]{0.2,0.6,0.2},
	urlcolor=[rgb]{0.6,0.2,0.2}}
\makeatother

\calclayout

\begin{document}

\maketitle

%\selectlanguage{french}

%\begin{abstract}
%	{\`A} tout système local motivique de Nori sur une variété algébrique complexe lisse et connexe, nous associons un lieu exceptionnel contrôlant la variation de complexité de ses fibres;
%	la définition est donnée en termes de groupes de Galois motiviques et motifs d'Artin.
%	Nous démontrons un analogue motivique du Théorème de Cattani--Deligne--Kaplan, selon lequel le lieu exceptionnel est une réunion dénombrable de sous-variétés algébriques fermées.
%	Nous montrons aussi que le lieu exceptionnel est défini sur tout sous-corps algébriquement clos sur lequel le système local motivique admet un modèle, et stable sous l'action de Galois lorsque ce denier descend à un sous-corps ultérieur.
%	Ceci renforce des résultats précédents d'André dans le cas pur.
%	Nous obtenons une description similaire pour le lieu de scindage de la filtration par le poids motivique.
%	Pour les systèmes locaux $1$-motiviques, ces propriétés passent aux variations de structure de Hodge sous-jacentes.
%\end{abstract}

\selectlanguage{english}

\begin{abstract}
To every Nori motivic local system over a smooth, connected complex algebraic variety, we associate an exceptional locus controlling the variation in the complexity of its stalks;
the definition is given explicitly in terms of motivic Galois groups and Artin motives. 
We prove a motivic analogue of the Cattani--Deligne--Kaplan Theorem, asserting that the exceptional locus is a countable union of closed algebraic subvarieties.
Moreover, we show that it is defined over any algebraically closed subfield over which the motivic local system admits a model, and stable under Galois conjugation when the latter descends to a smaller subfield.
This extends and strengthens previous results by André in the pure case.
We obtain a similar description for the splitting locus of the motivic weight filtration.
%For $1$-motivic local systems, these properties pass to the underlying variations of mixed Hodge structure.
\end{abstract}

\blfootnote{\textit{\subjclassname}: 14F42, 14C15, 14D07, 18M25.}
\blfootnote{\textit{\keywordsname}: Nori motives, motivic Galois groups, Hodge locus, weight filtration.}

\tableofcontents

\section*{Introduction}

\subsection*{Motivation and goal of the paper}

Let $X$ be a smooth, connected complex algebraic variety, and let $V$ be an admissible, graded-polarizable variation of mixed Hodge structure over $X$.
One can think of $V$ as a local system of mixed Hodge structures, and it is natural to ask how the complexity of the stalks $V_x$ varies with respect to $x \in X(\C)$.
This is measured by the (tensorial) Hodge locus of $V$, defined as the set of points $x \in X(\C)$ such that the mixed Hodge structure $V_x$ admits more Hodge tensors than the very general stalk.
A fundamental result describes such a locus as a countable union of strict closed algebraic subvarieties of $X$:
in the case of pure variations, this was proved in Cattani--Deligne--Kaplan's celebrated paper \cite{CDK95};
it was later extended to mixed variations by Brosnan--Pearlstein--Schnell in \cite{BPS10}, where a similar geometric description was provided for the locus over which the weight filtration splits.

The result is particularly interesting when the variation $V$ is of geometric origin -- that is, when it arises from the relative cohomology of $X$-schemes. 
In this case, its Hodge locus is expected to satisfy natural arithmetic properties, related to the Galois-theoretic nature of the corresponding $\ell$-adic local system.
This question, however, has been only partially answered to date.

A variation of geometric origin can be promoted to a Nori motivic local system, the stalks of which are Nori motives rather than mixed Hodge structures.
The main goal of the present paper is to define and study a natural motivic analogue of the (tensorial) Hodge locus, for which we establish unconditionally the same geometric and arithmetic properties.
In the pure case, our results extend and strengthen those obtained by André in his seminal paper on motivated cycles \cite{And96}. 
We also establish analogous properties for the splitting locus of the motivic weight filtration.
Our results are based on the Tannakian theory of Nori motivic local systems and motivic Galois groups, recently introduced by the author in \cite{Ter24Nori} and developed by Jacobsen in \cite{Jac25}, building on Ivorra--Morel's work \cite{IM24}.
In addition to showing the intrinsic robustness of this theory, they confirm its close relation to the theory of variations of mixed Hodge structure and open the way to further investigations, as discussed at the end of this introduction.  

%greatly developed thanks to the work of Ivorra--Morel, whose construction was initiated by Ivorra--Morel in \cite{IM24} and completed by the author in \cite{Ter24Nori}. 
%which has been greatly developed over the last few years, 

%A crucial ingredient is the motivic version of Deligne's theorem of the fixed part, established by Jacobsen in \cite{Jac25}.

\subsection*{History and previous work}

The study of the Hodge locus began with Weil's famous observation about Hodge classes in the relative cohomology of a smooth projective family over a smooth complex algebraic variety, as stated in \cite{Weil}. 
To fix notation, let $p \colon Y \to X$ be such a family, and let $n \geq 0$.
As $x \in X(\C)$ varies, the rational Hodge structures $V_x := H^{2n}(Y_x;\Q)(n)$ assemble into a polarizable pure variation $V$ of weight $0$ over $X$.
The \textit{Hodge locus} of $V$ is defined as the subset of points $x \in X(\C)$ such that the stalk $V_x$ contains more classes of type $(0,0)$ compared to the very general stalk.
Since the Hodge filtration on the stalks of $V$ varies holomorphically, this locus is a countable union of strict closed analytic subvarieties of $X$.
However, the Hodge Conjecture predicts that all rational $(0,0)$-classes on $V_x$ arise from algebraic cycles on the fibre $Y_x$. 
This would imply that the analytic subvarieties forming the Hodge locus are in fact algebraic.

This expectation was established, unconditionally on the Hodge Conjecture and for arbitrary polarizable pure variations, by Cattani--Deligne--Kaplan in \cite[Thm.~1.1]{CDK95}.
As later proved by Brosnan--Pearlstein--Schnell in \cite[Cor.~2]{BPS10}, the result remains true for arbitrary admissible, graded-polarizable variations of mixed Hodge structure:
this is deduced from the pure case via the theory of higher normal functions.
Using the same method, the authors obtained a similar geometric description for the splitting locus of the weight filtration.
All these results, and particularly the proof of \cite[Thm.~1.5]{CDK95}, are based on substantial input from Complex Analysis.

As explained by André in \cite[\S~2]{And92}, it is natural to consider Hodge classes not only on $V$ but on all of its tensor combinations at once, also known as Hodge tensors.
This leads one to replace the usual Hodge locus of $V$ by the \textit{tensorial Hodge locus}, defined as the set of points $x \in X(\C)$ such that $V_x$ contains more Hodge tensors than the very general stalk -- or, equivalently, such that the Mumford--Tate group of $V_x$ is smaller than that of the very general stalk.
The main advantage of the tensorial Hodge locus over the usual one is its amenability to group-theoretic techniques.

If the variation $V$ arises from a smooth projective family $p \colon Y \to X$, the Hodge locus (both in the usual and in the tensorial version) is expected to be defined over any algebraically closed subfield $k \subset \C$ over which the family admits a model, and stable under Galois conjugation when the latter descends to a further subfield:
again, this would follow from the validity of the Hodge Conjecture.
In contrast with the geometric structure of the Hodge locus, its arithmetic properties are not accessible via analytic methods, and they have not been established in full generality yet.
Important partial results have been obtained by work of Voisin \cite{Voisin}, Saito--Schnell \cite{SaitoSchnell}, and Klingler--Otwinowska--Urbanik \cite{KOU23} for variations defined over number fields:
in particular, the expected properties are known to hold for the maximal components of positive period dimension of the tensorial Hodge locus, under suitable hypotheses.
More recently, in \cite{KreutzEll} Kreutz defined the analogue of the tensorial Hodge locus for the $\ell$-adic local system corresponding to $V$ and established the same geometric and arithmetic properties unconditionally for it.
The Mumford--Tate Conjecture would imply that the tensorial Hodge locus agrees with this Galois-theoretic locus, thereby inheriting the arithmetic properties of the latter. 

Variations arising from smooth projective families are basic examples of \textit{variations of geometric origin}.
All such variations are realizations of finer cohomological objects known as \textit{Nori motivic local systems}, defined by the author in \cite[\S~6]{Ter24Nori}:
these can be thought of as local systems with values in the Tannakian category of Nori motives over $\C$, which is the best known candidate for the conjectural category of mixed motives.
Just as mixed Hodge structures are representations of Mumford--Tate groups, Nori motives are representations of \textit{motivic Galois groups}.
However, unlike the category of mixed Hodge structures, the category of Nori motives is defined in terms of a purely category-theoretic universal property.
This makes it often difficult to study its objects and compute their motivic Galois groups concretely.
In any case, it is natural to look for an analogue of the (tensorial) Hodge locus in the motivic setting.
For certain pure motivic local systems, this task was addressed by André in \cite[\S~5.2]{And96}, in the framework of pure motives built out of motivated cycles:
by \cite[Thm.~5.2]{And96}, the motivic Galois groups of the stalks stay constant away from an \textit{exceptional locus}, which is again a countable union of strict closed algebraic subvarieties of $X$.
André's proof realizes Weil's cycle-theoretic arguments in the motivic setting, bypassing the Hodge Conjecture.
This approach, however, does not generalize to arbitrary mixed motivic local systems, since they are not related to algebraic cycles as directly as the pure ones.
It is also unclear whether one can argue by induction starting from the pure case, as in the Hodge-theoretic setting.

\subsection*{Main results}

The main contribution of the present paper is to define and study the exceptional locus of an arbitrary mixed motivic local system.
To explain the general idea in precise terms, let again $X$ be a smooth, connected complex variety.
The Tannakian category $\MNoriLoc(X)$ of Nori motivic local systems over $X$ comes equipped with a monoidal forgetful functor
\begin{equation*}
	\iota_X \colon \MNoriLoc(X) \to \Loc(X)
\end{equation*} 
into the usual Tannakian category of local systems.
Thus, complex points $x \in X(\C)$ parametrize a natural family of fibre functors for $\MNoriLoc(X)$, and paths between points induce isomorphisms between the corresponding fibre functors.
For every $M \in \MNoriLoc(X)$, one gets a family of motivic Galois groups $\Gmot(M,x)$, defined as the Tannaka dual of $\langle M \rangle^{\otimes}$ at $x$:
these naturally assemble into a local system of algebraic groups.
The \textit{exceptional locus} $\MotLocus(M) \subset X(\C)$ is designed to measure how far are the subgroups $\Gmot(M_x) \hookrightarrow \Gmot(M,x)$ from forming a sub-local system.
More precisely, we regard $\Gmot(M_x)$ inside the normal subgroup
\begin{equation*}
	\Gmot^{\odot}(M,x) := \ker\left\{\Gmot(M,x) \twoheadrightarrow \Gmot^{\Artin}(M,x)\right\}.
\end{equation*} 
Here, $\Gmot^{\Artin}(M,x)$ denotes the Tannaka dual of the Artin motivic local systems inside $\langle M \rangle^{\otimes}$:
it is a finite group, conjecturally equal to the group of connected components of $\Gmot(M,x)$.
The subgroups $\Gmot^{\odot}(M,x)$ thus define again a local system of algebraic groups, conjecturally connected, and on can show that they contain indeed the subgroups $\Gmot(M_x)$ (see \Cref{lem:Gmot_fibre-to-Artin}).
We define the exceptional locus by declaring that $x \in \MotLocus(M)$ if and only if the inclusion $\Gmot(M_x) \hookrightarrow \Gmot^{\odot}(M,x)$ is not an isomorphism.
As in the Hodge-theoretic setting, the intuition is that the stalk $M_x$ is almost as complex as $M$ itself for most points $x \in X(\C)$, and this generic behavior drops precisely when $x \in \MotLocus(M)$.
However, due to the inexplicit nature of Nori motives, the mere definition of the exceptional locus does not suggest any evident constraints on its structure.
Nevertheless, our first main result shows that the expected properties are satisfied:
\begin{thm*}[\Cref{thm:MotLocus}, \Cref{prop:MotLocus-Gal}]
	For every $M \in \MNoriLoc(X)$, the following hold:
	\begin{enumerate}
		\item The locus $\MotLocus(M)$ is a countable union of subsets of $X(\C)$ of the form $S(\C)$, with $S$ a strict closed algebraic subvariety of $X$.
		\item Let $k \subset \C$ be a subfield such that $X$ admits a model $X_k$ over $k$ and $M$ admits a model in $\MNoriLoc(X_k)$, and let $\overline{k} \subset \C$ denote its algebraic closure.
		Then:
		\begin{enumerate}
			\item[(i)] All the maximal closed subvarieties of $X$ contained in $\MotLocus(M)$ are defined over $\overline{k}$.
			\item[(ii)] The natural $\Gal(\overline{k}/k)$-action on $X_{\overline{k}}$ stabilizes $\MotLocus(M)$. 
		\end{enumerate}
	\end{enumerate} 
\end{thm*}

Both the definition of the exceptional locus and the statement of the main theorem are inspired by André's \cite[Thm.~5.2]{And96}. 
Our treatment, however, builds on the robust structural theory of Nori motivic local systems:
this makes our definition more explicit and our arguments simpler.

While the first part of the theorem pertains to the geometry of the exceptional locus, the second part describes its arithmetic properties.
In the proof, however, we establish the first part together with assertion (i) in the second part at once:
the idea is that every motivic local system admits a model over some countable subfield $k \subset \C$, and $X$ contains only a countable amount of closed subvarieties defined over $\overline{k}$.
The key step is to understand the motivic Galois group $\Gmot(M_x)$ when $x$ is dense for the Zariski topology of $X_{\overline{k}}$, in which case we always have $x \notin \MotLocus(M)$.
To show this, we reinterpret motivic local systems over $X_k$ as motives over the function field $k(X_k)$, and we relate their stalks at Zariski-dense points to motives over the algebraic closure $\overline{k(X_k)}$.
This strategy relies on the invariance of motivic Galois groups under extensions of algebraically closed fields (see \Cref{prop:MNoriLoc_inv} and \Cref{cor:Gmot_inv}), which is an important result on its own.
Once the key step is achieved, the rest of the proof is a simple Noetherian induction argument.
We need to use Hironaka's classical work on Embedded Resolution of Singularities \cite{Hir64} to treat possibly singular closed subvarieties of $X_{\overline{k}}$.
To this end, we check in advance that the exceptional locus is compatible with birational morphisms (see \Cref{lem:MotLocus_birat}).

Assertion (ii) in the second part of the theorem is proved using similar ideas.
The only point that needs some care is the functoriality of motivic local systems and motivic Galois groups under $k$-automorphisms of $X_{\overline{k}}$ (see \Cref{constr:f^*-Gal}).
Once these technical details are settled, the proof comes down to the invariance of $M$ under Galois conjugation.
In fact, our argument shows that the motivic Galois groups at $\overline{k}$-rational points are constant along $\Gal(\overline{k}/k)$-orbits.  

Using the same methods, we are able to study the \textit{weight-splitting locus} $\WeightLocus(M) \subset X(\C)$.
Motivic local systems carry a canonical weight filtration, which agrees with the weight filtration on each stalk up to shifting suitably.
We define the weight-splitting locus by declaring that $x \in \WeightLocus(M)$ if and only if the motive $M_x$ is the direct sum of its weight-graded pieces -- or, equivalently, semi-simple in $\MNori(\C)$.
As in the previous case, the structure of the weight-splitting locus is not clear from the definition.
Our second main result describes it as follows:

\begin{thm*}[\Cref{thm:WeightLocus}, \Cref{prop:WeightLocus-Gal}]
	Let $M \in \MNoriLoc(X)$ be a motivic local system which is not semi-simple.
	Then, the following hold:
	\begin{enumerate}
		\item The locus $\WeightLocus(M)$ is a countable union of subsets of $\MotLocus(M)$ of the form $S(\C)$, with $S$ a strict closed algebraic subvariety of $X$.
		\item Let $k \subset \C$ be a subfield such that $X$ admits a model $X_k$ over $k$ and $M$ admits a model in $\MNoriLoc(X_k)$, and let $\overline{k} \subset \C$ denote its algebraic closure.
		Then:
		\begin{enumerate}
			\item[(i)] All the maximal closed subvarieties of $X$ contained in $\WeightLocus(M)$ are defined over $\overline{k}$.
			\item[(ii)] The natural $\Gal(\overline{k}/k)$-action on $X_{\overline{k}}$ stabilizes $\WeightLocus(M)$.
		\end{enumerate}
	\end{enumerate}
\end{thm*}

In fact, the same result holds for splitting loci of general short exact sequences of motivic local systems (see \Cref{thm:ext-splitting_locus} and \Cref{prop:SplitLocus-Gal}).
Compared to the proof of the previous result, the main new ingredient is the fact that semi-simplicity can be checked after quotienting out the contribution of Artin motives (see \Cref{prop:M_semisimple=Gmot0_reductive}).
As a by-product of our main results, we are able to characterize Artin and semi-simple motivic local systems in terms of their stalks (see \Cref{cor:Artin=trivial_stalk} and \Cref{cor:semi-simple_stalk}), and we obtain a similar characterization for split short exact sequences of motivic local systems (see \Cref{cor:split_stalk}). 

It is natural to wonder how close our motivic results are to the corresponding Hodge-theoretic results.
In particular, one might ask whether the results of the present paper could lead to a motivic proof of the main results of \cite{CDK95,BPS10} for variations of geometric origin.
Unfortunately, we are not able to say much about this point.
By the universal property of Nori's category of motives, there is a canonical monoidal Hodge realization functor
\begin{equation*}
	\iota_{\C}^{\Hodge} \colon \MNori(\C) \to \MHS
\end{equation*}
towards Deligne's category of graded-polarizable rational mixed Hodge structures.
Thanks to Tubach's recent work \cite{Tub25}, for any smooth complex variety $X$ there is an analogous functor
\begin{equation*}
	\iota_X^{\Hodge} \colon \MNoriLoc(X) \to \VHS(X)
\end{equation*}
towards the category of admissible, graded-polarizable variations of mixed Hodge structure.
Both are conjectured to be fully faithful, with essential image stable under subquotients.
Thus, given a motivic local system $M$ with underlying variation $V$, the motivic Galois group of $M_x$ is expected to coincide with the Mumford--Tate group of $V_x$ for every $x \in X(\C)$.
In particular, the exceptional locus of $M$ is expected to coincide with the tensorial Hodge locus of $V$;
and similarly for the weight-splitting loci.
Even if the above fullness conjecture remains widely open in general, its consequences about special loci might be easier to prove in principle.
However, for the time being, we are only able to observe that the motivic weight-splitting locus is contained inside the Hodge-theoretic one, as a consequence of \cite[Prop.~6.9]{Tub25}.
Concerning the exceptional locus, neither of the two inclusions seems evident.
In fact, we would be very surprised if it were possible to show that the two loci coincide just by comparing the definitions:
indeed, this would yield an alternative proof of the main results of \cite{CDK95,BPS10} in the geometric origin case with essentially no input from Complex Analysis, and it would also settle the difficult questions about fields of definition of (tensorial) Hodge loci.

Instead of considering the above problem in full generality, one can focus on Tannakian subcategories of motivic local systems for which the fullness conjecture is known to hold.
At present, the largest such Tannakian subcategory is the one generated by the $1$-motivic objects:
this has been shown by André in \cite[Thm.~1.2.1]{And21} for motives over a field, and extended by Tubach in \cite[Thm.~5.2.12]{TubThesis} to motivic local systems.
Thus, for $1$-motivic local systems and their tensor-combinations, the results of the present paper imply the analogous Hodge-theoretic results, including those about fields of definition.
To the best of our knowledge, all these results are already known, although we were unable to find an explicit reference:
indeed, $1$-motivic variations and their tensor-combinations can be obtained by inverse image from canonical variations over mixed Shimura varieties, and their rich arithmetic theory already subsumes all the geometric and arithmetic properties of special loci.
In fact, the theory of Hodge cycles on Shimura varieties is already used in the proof of \cite[Thm.~1.2.1]{And21} (which builds upon \cite[Thm.~0.6.2]{And96}).
On the other hand, it is worth stressing that the extension of the fullness conjecture from motives to motivic local systems is not specific to $1$-motivic objects, as it relies on general Tannakian arguments.
Therefore, the results of the present paper might have more interesting applications in the future, once new cases of the fullness conjecture are established.

\subsection*{Related and future work}

As already mentioned, our main results about the exceptional locus are inspired by André's \cite[Thm.~5.2]{And96}, formulated in the language of pure motives built out of motivated cycles.
As proved by Arapura in \cite[Thm.~6.4.1]{Ara13}, André's category of pure motives over $\C$ is canonically equivalent to the subcategory of semi-simple objects in $\MNori(\C)$.
André's result describes the variation in the motivic Galois groups of suitable families of pure motives (as defined in \cite[\S~5.2]{And96}):
the main examples are given by the relative cohomology of smooth projective families of complex varieties.
Every family of pure motives in André's sense defines a pure motivic local system in Nori's sense, with the same motivic Galois groups.
For such pure motivic local systems, our results are covered by \cite[Thm.~5.2]{And96}, although in a slightly less precise form.
However, we do not believe that every pure motivic local system over a general complex variety arises from a family of pure motives in André's sense.
Therefore, our results should be more general than André's even in the pure case.

The main advantage of working directly in the setting of Nori motives is the existence of a well-behaved Tannakian theory of motivic local systems over general varieties.
Firstly, our definition of the exceptional locus is more explicit than André's.
One could still define Tannakian categories of pure motivic local systems using André's language of motivated cycles, but only over function fields;
this seems to be of little help in order to study the exceptional locus.
Secondly, many of the results of \cite[\S~5.2]{And96} can be obtained more naturally within our setting.
For instance, in André's work, the fact that every motivic Galois group over $\C$ is realizable over some countable algebraically closed subfield is deduced from the main result about the exceptional locus.
In our work, by contrast, the same fact follows easily from the general theory of motivic local systems --
and, as explained above, it is used crucially in our study of the exceptional locus.

Let us point out that the general ideas behind the proof of our main results are very similar to those of Kreutz's paper \cite{KreutzEll}, where analogous results are obtained in the setting of $\ell$-adic local systems.
While we were not aware of Kreutz's work during the preparation of the present paper, we convinced ourselves that they share indeed the same formal structure:
the key properties of motivic Galois groups used here mimic well-known properties of étale fundamental groups used there.
Kreutz's definition of the $\ell$-adic exceptional locus, however, involves unavoidably the fields of definition of closed subvarieties in the ambient complex variety.
By contrast, our definition of the motivic exceptional locus is closer in spirit to the original Hodge-theoretic notion, as it takes place purely in the setting of complex varieties:
it is the proof of our main result that unveils its arithmetic nature.

Kreutz's papers \cite{KreutzAbs,KreutzEll} also investigate the analogue of the tensorial Hodge locus in the setting of motives for absolute Hodge cycles.
A refined version of the Mumford--Tate Conjecture predicts that this absolute tensorial Hodge locus should coincide both with the classical tensorial Hodge locus and with the $\ell$-adic exceptional locus
By \cite[Thm.~3.4]{KreutzEll}, this conclusion holds unconditionally for the part of positive period dimension, under suitable hypotheses. 
It will be interesting to see whether a similar comparison result can be proved about our motivic exceptional locus. 

More generally, we plan to refine our understanding of the motivic exceptional locus in future work, inspired by the remarkable developments in the study of the Hodge locus over the past few years.
The tensorial Hodge locus of positive period dimension is further analyzed in Klingler--Otwinowska's paper \cite{KO21}, where it is shown that, under suitable hypotheses, it satisfies a neat dichotomy:
it is either Zariski-dense or algebraic.
As further shown in Baldi--Klingler--Ullmo's work \cite{BKU}, it is the level of a variation that dictates which alternative holds.
In the case of variations defined over number fields, the abundance of non-exceptional points with algebraic coordinates has been recently established in Baldi--Binyamini--Urbanik's paper \cite{BaBiUr}.
Their methods also settle new cases of the Mumford--Tate Conjecture beyond the range of $1$-motives.
In principle, the same methods could be helpful to attack new cases of the fullness conjecture for the Hodge realization of Nori motives. 

Due to the inexplicit definition of Nori's category, motivic Galois groups remain more mysterious than Mumford--Tate groups.
For instance, while it is known that Mumford--Tate groups of mixed Hodge structures are detected by their invariant tensors (see \cite[Prop.~3.1(c), Rmk.~3.2(b)]{DM82Ab} or \cite[Lem.~2(a)]{And92}), establishing the same property for motivic Galois groups is an open problem.
Thus, our motivic exceptional locus does not admit an immediate interpretation in terms of exceptional motivic tensors on the stalks of a motivic local system.
It will be interesting to study this question more thoroughly, for example by considering a motivic analogue of the usual (as opposed to tensorial) Hodge locus and adapting the main results of the present paper accordingly.

\subsection*{Structure of the paper}

The basic definitions and results about Nori motives, motivic local systems, and motivic Galois groups are summarized in \Cref{sect:MNori}.
In \Cref{sect:Gmot-wow}, they are complemented by more specific results relating different motivic Galois groups;
some of them are already known, although not always available in the form needed for our purposes in the literature.
The core of the paper is \Cref{sect:MotLocus}, where we define the exceptional locus and establish its main geometric and arithmetic properties.
Following the same route, in \Cref{sect:WeightLocus} we study the splitting locus of the motivic weight filtration and, more generally, splitting loci of motivic short exact sequences.
%The final \Cref{sect:appli_1mot} describes our applications to variations of Hodge structure underlying $1$-motivic local systems.

\subsection*{Acknowledgments}

I am grateful to David Urbanik for his detailed answers to my questions about Hodge loci and $1$-motivic variations. 
I would also like to thank Joseph Ayoub, Annette Huber-Klawitter, and Emil Jacobsen for useful conversations.

The author was supported by the Deutsche Forschungsgemeinschaft through the Walter Benjamin Programme, project number 568083338.

\section*{Notation and conventions}

\subsection*{Tannakian categories}
\begin{itemize}
	\item Our conventions on Tannakian categories follow those of \cite{DM82}.
	All Tannakian categories considered in the present paper are neutral Tannakian over $\Q$.
	\item By a \textit{Tannakian subcategory} of a given Tannakian category $\Tannakian$ we mean a strictly full abelian subcategory of $\Tannakian$ which contains the monoidal unit of $\Tannakian$ and is stable under subquotients, tensor products, and duals:
	endowed with the induced monoidal structure, it is itself a Tannakian category.
	\item Given a Tannakian category $\Tannakian$, for every object $A \in \Tannakian$ we let $\langle A \rangle^{\otimes}$ denote the Tannakian subcategory of $\Tannakian$ \textit{generated} by $A$:
	that is, the intersection of all Tannakian subcategories of $\Tannakian$ containing $A$.
	\item Given two Tannakian categories $\Tannakian_1$ and $\Tannakian_2$, by a \textit{monoidal functor} $F \colon \Tannakian_1 \to \Tannakian_2$ we mean an exact functor endowed with compatibility isomorphisms with respect to the tensor product respecting associativity, commutativity, and unit constraints (as defined in \cite[\S~1]{DM82}).
\end{itemize}

\subsection*{Algebraic and Analytic Geometry}
\begin{itemize}
	\item We typically work over a field $k$ of characteristic $0$ endowed with a complex embedding $\sigma \colon k \hookrightarrow \C$.
	When $k = \C$, we choose $\sigma$ to be the identity.
	\item By a \textit{$k$-variety} we mean a reduced, separated $k$-scheme of finite type.
	By a morphism of $k$-varieties we always mean a $k$-morphism.
	\item For every $k$-variety $X$, we let $a_X \colon X \to \Spec(k)$ denote the structural morphism.
	\item Fix a complex embedding $\sigma \colon k \hookrightarrow \C$. 
	For every $k$-variety $X$, we use the following notation:
	\begin{itemize}
		\item We let $X^{\sigma}$ denote the associated complex-analytic space:
		that is, the space of complex points $(X \times_{k,\sigma} \C)(\C)$ endowed with the complex-analytic topology.
		When $k = \C$, we write $X(\C)$ in place of $X^{\sigma}$.
		\item We let $D^b_{\sigma}(X)$ denote the stable $\infty$-category of $k$-algebraically constructible complexes with $\Q$-coefficients over $X^{\sigma}$.
        Moreover, we let $\Perv_{\sigma}(X)$ denote the heart of the perverse $t$-structure on $D^b_{\sigma}(X)$ (as defined in \cite[\S~2]{BBD82}).
		\item If $X$ is smooth and geometrically connected over $k$, we write $\Loc_{\sigma}(X)$ for the Tannakian category of local systems with $\Q$-coefficients over $X^{\sigma}$, and we always identify it with a full abelian subcategory of $\Perv_{\sigma}(X)$ (as in \cite[\S~4]{BBD82}).
		For every point $x \in X^{\sigma}$, we let $\pi_1^{\alg}(X^{\sigma},x)$ denote the Tannaka dual of $\Loc_{\sigma}(X)$ with respect to the fibre functor at $x$.
		Similarly, given $L \in \Loc_{\sigma}(X)$, we write $\pi_1^{\alg}(L,x)$ for the Tannaka dual of $\langle L \rangle^{\otimes}$ with respect to the same fibre functor.
	\end{itemize}  
\end{itemize}

\subsection*{Nori motivic sheaves}
\begin{itemize}
	\item For every $k$-variety $X$, we let $\MNori_{\sigma}(X)$ denote the abelian category of Nori motivic perverse sheaves over $X$ (as defined in \cite[\S~2.1]{IM24}).
	When $X = \Spec(k)$, we write $\MNori_{\sigma}(k)$ in place of $\MNori_{\sigma}(\Spec(k))$.
	\item If $X$ is smooth and geometrically connected over $k$, we let $\MNoriLoc_{\sigma}(X)$ denote the Tannakian category of Nori motivic local systems over $X$ (as defined in \cite[\S~6]{Ter24Nori}).
	For every point $x \in X^{\sigma}$, we write $\Gmot(X,x;\sigma)$ for the Tannaka dual of $\MNoriLoc_{\sigma}(X)$ with respect to the fibre functor at $x$.
	Similarly, given $M \in \MNoriLoc_{\sigma}(X)$, we write $\Gmot(M,x;\sigma)$ for the Tannaka dual of $\langle M \rangle^{\otimes}$ with respect to the same fibre functor.
\end{itemize}

%\subsection*{Mixed Hodge modules}
%\begin{itemize}
%	\item Let again $X$ be a $k$-variety.
%	\begin{itemize}
%		\item We write $\MHM_{\sigma}(X)$ for the abelian category of algebraic mixed Hodge modules over the complex variety $X \times_{k,\sigma} \C$ (as defined in \cite{Sai90}). 
%	\end{itemize}
%	\item Suppose that $X$ is smooth an geometrically connected over $k$.
%	\begin{itemize}
%		\item We let $\VHS_{\sigma}(X)$ denote the abelian subcategory of smooth mixed Hodge modules over $X$ --
%		that is, the admissible, graded-polarizable variations of mixed Hodge structure.
%		\item For every point $x \in X^{\sigma}$, we let $\GHodge(X^{\sigma},x)$ denote the Tannaka dual of $\VHS_{\sigma}(X)$ with respect to the fibre functor at $x$.
%		\item Similarly, for every $V \in \VHS_{\sigma}(X)$, we let $\GHodge(V,x)$ denote the Tannaka dual of $\langle V \rangle^{\otimes}$ with respect to the fibre functor at $x$.
%	\end{itemize}
%\end{itemize}

\section{Recollections on Nori motivic local systems}\label{sect:MNori}

In this section, we review Nori's theory of mixed motives and its extension to a theory of mixed motivic sheaves, with a focus on motivic local systems and motivic Galois groups.
The reader already familiar with the theory can safely skip this brief account.

Throughout, we work over a field $k$ of characteristic $0$ endowed with a complex embedding $\sigma \colon k \hookrightarrow \C$.
%In the last part, we discuss the Hodge realization of Nori motives.

\subsection{Nori motives}\label{subsect:MNori(k)}

We begin by reviewing Nori's theory of motives over $k$.
For convenience, we do not follow Nori's original approach (as discussed in \cite[\S~II]{HMS17}) but rather Ivorra--Morel's modern approach (as developed in \cite[\S\S~1-2]{IM24}).

The category $\MNori_{\sigma}(k)$ of \textit{Nori motives} over $k$ is defined as the so-called universal abelian factorization of the homological functor
\begin{equation*}
	\beta_{\sigma,k} \colon \DAct(k) \xrightarrow{\Bti_{\sigma,k}^*} D^b(\vect_{\Q}) \xrightarrow{H^0} \vect_{\Q}.
\end{equation*}
Here, $\DAct(k)$ is Voevodsky's triangulated category of constructible motives over $k$ (as defined in Ayoub's work \cite{Ayo07a,Ayo07b}) and $\Bti_{\sigma,k}^* \colon \DAct(k) \to D^b(\vect_{\Q})$ denotes the Betti realization associated to $\sigma$ (see \cite[Defn.~2.1]{Ayo10}). 
Intuitively, $\DAct(k)$ contains the universal cohomology complexes of all $k$-varieties, and the Betti realization converts them into the usual singular cohomology complexes.
By construction, Nori's abelian category comes equipped with a homological functor
\begin{equation*}
	\pi_{\sigma,k} \colon \DAct(k) \to \MNori_{\sigma}(k)
\end{equation*}
and with a faithful exact functor
\begin{equation*}
	\iota_{\sigma,k} \colon \MNori_{\sigma}(k) \to \vect_{\Q}
\end{equation*}
providing a universal factorization of $\beta_{\sigma,k}$ of the form
\begin{equation*}
	\beta_{\sigma,k} \colon \DAct(k) \xrightarrow{\pi_{\sigma,k}} \MNori_{\sigma}(k) \xrightarrow{\iota_{\sigma,k}} \vect_{\Q}.
\end{equation*}
%Namely, $\MNori_{\sigma}(k)$ is the initial abelian category providing such a factorization, in a precise technical sense (see \cite[\S 2.1]{IM24} for more details).
Intuitively, $\MNori_{\sigma}(k)$ contains the single universal cohomology groups of all $k$-varieties.

Given a field extension $k'/k$ and a complex embedding $\sigma' \colon k' \hookrightarrow \C$ extending $\sigma$, the Betti realization functors fit into a commutative diagram of the form
\begin{equation*}
	\begin{tikzcd}
		\DAct(k) \arrow{rr} \arrow{dr}{\Bti_{\sigma,k}^*} &&
		\DAct(k') \arrow{dl}{\Bti_{\sigma',k'}^*} \\
		&
		D^b(\vect_{\Q}),
	\end{tikzcd}
\end{equation*}
where the horizontal arrow denotes the base-change functor on Voevodsky motives (for example, see \cite[Lem.~1.6]{JT25}). 
This induces a canonical faithful exact base-change functor
\begin{equation*}
	\MNori_{\sigma}(k) \to \MNori_{\sigma'}(k')
\end{equation*}
rendering the diagram
\begin{equation}\label{dia:MNori_k'/k}
	\begin{tikzcd}
		\MNori_{\sigma}(k) \arrow{rr} \arrow{dr}{\iota_{\sigma,k}} &&
		\MNori_{\sigma'}(k') \arrow{dl}{\iota_{\sigma',k'}} \\
		&
		\vect_{\Q}
	\end{tikzcd}
\end{equation}
commutative up to natural isomorphism. 
This construction is compatible with composition of field extensions.
The following result will be useful later on:

\begin{lem}[{\cite[Prop.~6.9]{IM24}}]\label[lem]{lem:MNori_colim}
	Let $\left\{k_i\right\}_{i \in I}$ be a filtered family of subfields of $k$ such that $k = \bigcup_{i \in I} k_i$; 
	for every $i \in I$, set $\sigma_i := \sigma|_{k_i}$.
	Then, the induced functor
	\begin{equation*}
		\twocolim_{i \in I} \MNori_{\sigma_i}(k_i) \rightarrow \MNori_{\sigma}(k)
	\end{equation*}
	is an equivalence.
\end{lem}
\begin{proof}
	Since universal abelian factorizations are compatible with filtered $2$-colimits (see \cite[Lem.~5.12]{Ter24UAF}), this follows from the fact that the analogous functor of Voevodsky motives
	\begin{equation*}
		\twocolim_{i \in I} \DAct(k_i) \to \DAct(k)
	\end{equation*}
	is an equivalence (see \cite[Cor.~3.22]{Ayo14Et}).
\end{proof}

The most remarkable aspect of Nori's theory is the existence of a tensor product on $\MNori_{\sigma}(k)$ compatible with the usual tensor product of singular cohomology groups.
More precisely:
\begin{thm}[Nori, {\cite[Thm.~9.3.10]{HMS17}}]\label[thm]{thm:Nori_tens}
	The abelian category $\MNori_{\sigma}(k)$ carries a canonical unitary symmetric monoidal structure, which turns it into a neutral Tannakian category over $\Q$ with canonical fibre functor $\iota_{\sigma,k} \colon \MNori_{\sigma}(k) \to \vect_{\Q}$.
\end{thm}

\begin{nota}
	We let $\Gmot(k;\sigma)$ denote the Tannaka dual of $\MNori_{\sigma}(k)$ with respect to the fibre functor $\iota_{\sigma,k}$, and we call it the \textit{motivic Galois group} of $k$.
	Similarly, for every $N \in \MNori_{\sigma}(k)$, we let $\Gmot(N;\sigma)$ denote the Tannaka dual of the Tannakian subcategory $\langle N \rangle^{\otimes} \subset \MNori_{\sigma}(k)$ with respect to the same fibre functor, and we call it the \textit{motivic Galois group} of $N$.
\end{nota}

%Since the full motivic Galois group is quite inaccessible, it is natural to study it through its finite-dimensional quotients.
Under Tannaka duality, the inclusion $\langle N \rangle^{\otimes} \subset \MNori_{\sigma}(k)$ corresponds to a surjection
\begin{equation*}
	\Gmot(k;\sigma) \twoheadrightarrow \Gmot(N;\sigma).
\end{equation*}
Every finite-dimensional quotient of the pro-algebraic group $\Gmot(k;\sigma)$ is of the form $\Gmot(N;\sigma)$ for some $N \in \MNori_{\sigma}(k)$.

\begin{ex}
	The \textit{unit motive} over $k$ is defined as 
	\begin{equation*}
		\Q_k := \pi_{\sigma,k}(\unit_k) \in \MNori_{\sigma}(k),
	\end{equation*}
	where $\unit_k$ denotes the monoidal unit of $\DAct(k)$.
	As usual, the fibre functor induces an equivalence $\iota_{\sigma,k} \colon \langle \Q_k \rangle^{\otimes} \xrightarrow{\sim} \vect_{\Q}$.
	We call an object $N \in \MNori_{\sigma}(k)$ \textit{trivial} if $N \in \langle \Q_k \rangle^{\otimes}$ -- or, equivalently, if $\Gmot(N;\sigma) = 1$.
	\qed
\end{ex}

The easiest examples of non-trivial Nori motives arise from finite extensions of the base field $k$.
\begin{ex}[{\cite[\S~9.4]{HMS17}}; see also {\cite[Ex.~2.8]{JT25}}]\label[ex]{ex:Artin_k}
	Let $k'/k$ be a finite field extension, and let $e \colon \Spec(k') \to \Spec(k)$ denote the corresponding finite étale morphism.
	The associated \textit{Artin motive} is defined as
	\begin{equation*}
		\Q_{k'/k} := \pi_{\sigma,k}(e_* \unit_{k'}) \in \MNori_{\sigma}(k).
	\end{equation*}
	If the extension $k'/k$ is Galois, the motivic Galois group $\Gmot(\Q_{k'/k};\sigma)$ is canonically isomorphic to the Galois group $\Gal(k'/k)$, regarded as a finite constant algebraic group over $\Q$.
	Given a further finite extension $k''/k'$ such that the composite extension $k''/k$ is again Galois, we have the inclusion $\langle \Q_{k'/k} \rangle^{\otimes} \subset \langle \Q_{k''/k} \rangle^{\otimes}$ inside $\MNori_{\sigma}(k)$.
	Under Tannaka duality, it corresponds to the usual quotient homomorphism 
	\begin{equation*}
		\Gal(k''/k) \twoheadrightarrow \Gal(k'/k),
	\end{equation*}
	regarded as a morphism of algebraic groups over $\Q$.
	\qed
\end{ex}

\begin{nota}
	We let $\MNori_{\sigma}^{\Artin}(k)$ denote the category of \textit{Artin motives} over $k$, defined as the union of the subcategories $\langle \Q_{k'/k} \rangle^{\otimes} \subset \MNori_{\sigma}(k)$ as $k'$ varies among all finite extensions of $k$ inside a fixed algebraic closure $\overline{k}$:
	it is again a Tannakian subcategory of $\MNori_{\sigma}(k)$.
\end{nota}

Recall that, under Tannaka duality, filtered unions of Tannakian subcategories inside an ambient Tannakian category correspond to inverse limits of Tannaka-dual groups.
In the case of Artin motives, this means that the Tannaka dual of $\MNori_{\sigma}^{\Artin}(k)$ is canonically isomorphic to the absolute Galois group $\Gal(\overline{k}/k)$, regarded as a pro-finite algebraic group over $\Q$.
Thus, the inclusion $\MNori_{\sigma}^{\Artin}(k) \subset \MNori_{\sigma}(k)$ determines a canonical surjection
\begin{equation}\label{eq:Gmot-to-Gal}
	\Gmot(k;\sigma) \twoheadrightarrow \Gal(\overline{k}/k).
\end{equation}
Given a field extension $k'/k$ and a complex embedding $\sigma' \colon k' \hookrightarrow \C$ extending $\sigma$, the base-change functor $\MNori_{\sigma}(k) \to \MNori_{\sigma'}(k')$ is canonically monoidal (see \cite[\S~9.5]{HMS17}), and so is the natural isomorphism filling the diagram \eqref{dia:MNori_k'/k}.
This determines a homomorphism of Tannaka-dual groups
\begin{equation*}
	\Gmot(k';\sigma') \to \Gmot(k;\sigma),
\end{equation*}
which in general is neither surjective nor a closed immersion.

The situation is easiest to understand in the case of an algebraic extension.
Until the end of the present subsection, fix an algebraic closure $\overline{k}$ of $k$ as well as a complex embedding $\overline{\sigma} \colon \overline{k} \hookrightarrow \C$ extending $\sigma$.
The base-change functor $\MNori_{\sigma}(k) \to \MNori_{\overline{\sigma}}(\overline{k})$ corresponds to a homomorphism
\begin{equation}\label{eq:Gmotbar-to-Gmot}
	\Gmot(\overline{k};\overline{\sigma}) \to \Gmot(k;\sigma),
\end{equation}
and we have the following fundamental relation:
\begin{prop}[{\cite[Thm.~9.1.16]{HMS17}};
	see also {\cite[Prop.~2.14]{JT25}}]\label[prop]{prop:ses_Gmot-Gmot-Gal}
    The sequence of pro-algebraic groups
	\begin{equation*}
		1 \to \Gmot(\overline{k};\overline{\sigma}) \to \Gmot(k;\sigma) \to \Gal(\overline{k}/k) \to 1
	\end{equation*}
	defined by the homomorphisms \eqref{eq:Gmotbar-to-Gmot} and \eqref{eq:Gmot-to-Gal} is exact.
\end{prop}

For our purposes, it is convenient to work with a finite-dimensional variant of this result.

\begin{nota}
	For every $N \in \MNori_{\sigma}(k)$, we let $\langle N \rangle^{\otimes}_{\Artin}$ denote the intersection $\langle N \rangle^{\otimes} \cap \MNori_{\sigma}^{\Artin}(k)$ inside $\MNori_{\sigma}(k)$:
	it is again a Tannakian subcategory of $\MNori_{\sigma}(k)$.
	Moreover, we let $\Gmot^{\Artin}(N;\sigma)$ denote the Tannaka dual of $\langle N \rangle^{\otimes}_{\Artin}$ with respect to the fibre functor $\iota_{\sigma,k}$.
\end{nota}

\begin{rem}\label[rem]{rem:Gmot^Artin-finite}
	 By construction, $\Gmot^{\Artin}(N;\sigma)$ is the maximal quotient of $\Gmot(k;\sigma)$ which factors through both $\Gmot(N;\sigma)$ and $\Gal(\overline{k}/k)$.
	 Being an algebraic quotient of the pro-finite group $\Gal(\overline{k}/k)$, it is a finite group, equal to $\Gal(k'/k)$ for a unique finite Galois extension $k'/k$.
\end{rem}

\begin{cor}\label[cor]{cor:ses_Gmot-Artin}
	Fix an object $N \in \MNori_{\sigma}(k)$, and let $\overline{N}$ denote its image in $\MNori_{\overline{\sigma}}(\overline{k})$.
	Then, the sequence of algebraic groups
	\begin{equation*}
		1 \to \Gmot(\overline{N};\overline{\sigma}) \to \Gmot(N;\sigma) \to \Gmot^{\Artin}(N;\sigma) \to 1
	\end{equation*}
	is exact.
\end{cor}
\begin{proof}
	Right-exactness follows formally from the commutative diagram with surjective vertical arrows
	\begin{equation*}
		\begin{tikzcd}
			\Gmot(\overline{k};\overline{\sigma}) \arrow{r} \arrow[two heads]{d} &
			\Gmot(k;\sigma) \arrow{r} \arrow[two heads]{d} &
			\Gal(\overline{k}/k) \arrow{r} \arrow[two heads]{d} &
			1 \\
			\Gmot(\overline{N};\overline{\sigma}) \arrow{r} &
			\Gmot(N;\sigma) \arrow{r} &
			\Gmot^{\Artin}(N;\sigma) \arrow{r} & 
			1,
		\end{tikzcd}
	\end{equation*}
	in which the upper row is exact (by \Cref{prop:ses_Gmot-Gmot-Gal}).
	To conclude, it remains to show that the arrow $\Gmot(\overline{N};\overline{\sigma}) \to \Gmot(N;\sigma)$ is a closed immersion.
	This is equivalent to saying that every object of $\langle \overline{N} \rangle^{\otimes}$ is a subquotient of one in the essential image of $\langle N \rangle^{\otimes}$ (by \cite[Prop.~2.21(b)]{DM82}), which is clear.
\end{proof}

\begin{rem}\label[rem]{rem:Gmot(kbar)=connected}
	The pro-algebraic group $\Gmot(\overline{k};\overline{\sigma})$ is conjectured to be connected (see \cite[Rmk.~9.5.7]{HMS17}).
	If this is the case, the same holds for each algebraic quotient $\Gmot(\overline{N};\overline{\sigma})$.
	In particular, the finite group $\Gmot^{\Artin}(N;\sigma)$ is expected to coincide with the group of connected components of $\Gmot(N;\sigma)$.
\end{rem}

\begin{cor}\label[cor]{cor:N'_to_Nbar_fullyfaith}
	Let $N$ and $\overline{N}$ be as in \Cref{cor:ses_Gmot-Artin}.
	Moreover, let $k'/k$ denote the finite Galois extension with Galois group $\Gmot^{\Artin}(N;\sigma)$ (as in \Cref{rem:Gmot^Artin-finite}); 
	write $\sigma' := \overline{\sigma}|_{k'}$, and let $N'$ denote the image of $N$ in $\MNori_{\sigma'}(k')$.
	Then, the base-change functor
	\begin{equation*}
		\langle N' \rangle^{\otimes} \to \langle \overline{N} \rangle^{\otimes}
	\end{equation*}
	is an equivalence. 
\end{cor}
\begin{proof}
	The thesis is equivalent to the assertion that the closed immersion
	\begin{equation*}
		\Gmot(\overline{N};\overline{\sigma}) \hookrightarrow \Gmot(N';\sigma')
	\end{equation*}
	is in fact an isomorphism.
	To this end, in view of \Cref{cor:ses_Gmot-Artin}, it suffices to show that the sequence
	\begin{equation*}
		1 \to \Gmot(N';\sigma') \to \Gmot(N;\sigma) \to \Gmot^{\Artin}(N;\sigma) \to 1
	\end{equation*}
	is exact. 
	But injectivity of the arrow $\Gmot(N';\sigma') \to \Gmot(N;\sigma)$ follows as in the proof of \Cref{cor:ses_Gmot-Artin}, while exactness in the middle follows from the very definition of the extension $k'/k$.
\end{proof}

\begin{rem}
	There exists an alternative approach to motivic Galois groups, directly within the theory of Voevodsky motives:
	it has been developed by Ayoub in \cite{Ayo14H1,Ayo14H2} and compared to Nori's approach in Choudhury--Gallauer's paper \cite{CG17}.
	However, we do not need to adopt this alternative perspective in the present paper.
\end{rem}

\subsection{Nori motivic sheaves}\label{subsect:MNori(X)}

The definition of Nori motives as a universal abelian factorization can be generalized over arbitrary $k$-varieties:
the abelian category $\MNori_{\sigma}(X)$ of \textit{Nori motivic perverse sheaves} over $X$ is defined in \cite[\S~2.1]{IM24} as the universal abelian factorization of the homological functor
\begin{equation*}
	\beta_{\sigma,X} \colon \DAct(X) \xrightarrow{\Bti_{\sigma,X}^*} D^b_{\sigma}(X) \xrightarrow{{^p H^0}} \Perv_{\sigma}(X),
\end{equation*}
where $\Bti_{\sigma,X}^* \colon \DAct(X) \to D^b_{\sigma}(X)$ denotes the Betti realization of Voevodsky's constructible motivic sheaves over $X$ (see \cite[Defn.~2.1]{Ayo10}).
In the case when $X = \Spec(k)$, one obtains the category of Nori motives introduced in the previous subsection.

Several basic properties of Nori motives hold for motivic perverse sheaves as well.
For instance, we have the following generalization of \Cref{lem:MNori_colim}:
\begin{lem}\label[lem]{lem:MNori(X)_colim}
	Keep the notation and assumptions of \Cref{lem:MNori_colim}.
	Given a $k$-variety $X$, choose an index $i_0 \in I$ such that $X$ admits a model $X_0$ over $k_{i_0}$;
	for every $i \in I_{i_0/}$, set $X_i := X_0 \times_{k_{i_0}} k_i$.
	Then, the induced functor
	\begin{equation*}
		\twocolim_{i \in I} \MNori_{\sigma_i}(X_i) \rightarrow \MNori_{\sigma}(X)
	\end{equation*}
	is an equivalence.
\end{lem}
\begin{proof}
	As in the proof of \Cref{lem:MNori_colim}, this follows from the fact that the analogous functor of Voevodsky motivic sheaves
	\begin{equation*}
		\twocolim_{i \in I} \DAct(X_i) \to \DAct(X)
	\end{equation*}
	is an equivalence (see \cite[Cor.~3.22]{Ayo14Et}).
\end{proof}

The main structural result in the theory of Nori motivic sheaves is the following:
\begin{thm}[{\cite[Thm.~5.1]{IM24}}; {\cite[Thm.~2.1, Thm.~5.1]{Ter24Nori}}; {\cite[Constr.~3.15, Prop.~A.5]{Tub25}}]\label[thm]{thm:Nori-6ff}
	As $X$ varies among $k$-varieties, the derived categories $D^b(\MNori_{\sigma}(X))$ are endowed with a canonical six-functor formalism, and the forgetful functors
	\begin{equation*}
		D^b(\MNori_{\sigma}(X)) \xrightarrow{\iota_X} D^b(\Perv_{\sigma}(X)) \xrightarrow{\sim} D^b_{\sigma}(X)
	\end{equation*}
	commute with the six operations.
\end{thm}
The last passage in the equation witnesses Beilinson's equivalence, as in \cite[Thm.~1.3]{Beilinson}.
Note that the result incorporates Nori's \Cref{thm:Nori_tens}, which gives the monoidal structure over $\Spec(k)$.

Using the six operations, one can easily relate different categories of motivic perverse sheaves in the expected way.
As an example:
\begin{lem}\label[lem]{lem:MNori(X)=MNori(X')^Gal}
	Let $e \colon X' \to X$ be a finite étale morphism.
	Then, the inverse image functor $e^* \colon \MNori_{\sigma}(X) \to \MNori_{\sigma}(X')$ induces an equivalence 
	\begin{equation*}
		\MNori_{\sigma}(X) \to \MNori_{\sigma}(X')^{\Gal(X'/X)}
	\end{equation*}
	towards the $\Gal(X'/X)$-equivariant objects with respect to the $\Gal(X'/X)$-action on $\MNori_{\sigma}(X')$ induced by the natural Galois action on $X'$.
\end{lem}
\begin{proof}
	In the special case when $e$ is induced by a finite field extension $k'/k$, the result is \cite[Lem.~2.13(2)]{JT25}:
	the proof comes down to the fact that motivic perverse sheaves form a stack for the étale topology (by \cite[Prop.~2.7]{IM24}).
	The same argument works in the general case considered here.
\end{proof}

\begin{rem}\label[rem]{rem:MNori_sigma(k')=MNori_sigma'(k')}
	Suppose that $e$ is induced by a finite extension $k'/k$, and fix a complex embedding $\sigma' \colon k' \hookrightarrow \C$ extending $\sigma$.
	We can regard $X'$ either as a $k'$-variety or as a $k$-variety.
	However, the corresponding categories $\MNori_{\sigma'}(X')$ and $\MNori_{\sigma}(X')$ are canonically equivalent (by \cite[Prop.~6.11]{IM24}, see also \cite[Prop.~1.14]{Ter24Nori}).
\end{rem}

\subsection{Nori motivic local systems}

Let now $X$ be a smooth, connected $k$-variety.
The shift endofunctor $[\dim(X)] \colon D^b_{\sigma}(X) \to D^b_{\sigma}(X)$ identifies the abelian category of ordinary local systems over $X^{\sigma}$ with a full abelian subcategory $\Loc_{\sigma}(X) \subset \Perv_{\sigma}(X)$, stable under subquotients and extensions (for example, see \cite[Lem.~1.29]{JT25}).
This naturally generalizes to arbitrary smooth $k$-varieties, by considering one connected component at a time.
In the present paper, we always regard local systems as perverse sheaves in this way.

We are mostly interested in the case when $X$ is geometrically connected over $k$ --
or, equivalently, when the complex-analytic space $X^{\sigma}$ is connected (see \cite[Exp.~XII, Cor.~2.6]{SGA1}).
In this case, ordinary local systems over $X^{\sigma}$ form a neutral Tannakian category over $\Q$ with respect to the usual tensor product.
With our conventions, this corresponds to the shifted tensor product
\begin{equation*}
	- \otimes^{\dagger} - := (- \otimes -)[-\dim(X)] \colon \Loc_{\sigma}(X) \times \Loc_{\sigma}(X) \to \Loc_{\sigma}(X).
\end{equation*}
The Tannaka dual of $\Loc_{\sigma}(X)$ at a point $x \in X^{\sigma}$ is the algebraic fundamental group $\pi_1^{\alg}(X^{\sigma},x)$, defined as the $\Q$-algebraic completion of the usual topological fundamental group $\pi_1(X^{\sigma},x)$.

For every morphism $f \colon X \to Y$ between smooth, connected $k$-varieties, the shifted inverse image functor
\begin{equation*}
	f^{\dagger} := f^*[\dim(X) - \dim(Y)] \colon D^b_{\sigma}(Y) \to D^b_{\sigma}(X)
\end{equation*}
restricts to a monoidal functor
\begin{equation*}
	f^{\dagger} \colon \Loc_{\sigma}(Y) \to \Loc_{\sigma}(X).
\end{equation*} 
This is our model for the theory of Nori motivic local systems.
In the following, let $X$ be a smooth, geometrically connected $k$-variety.

\begin{nota}[{\cite[Defn.~6.1]{Ter24Nori}}]
	An object $M \in \MNori_{\sigma}(X)$ is called a \textit{motivic local system} if its underlying perverse sheaf $\iota_X(M) \in \Perv_{\sigma}(X)$ belongs to $\Loc_{\sigma}(X)$. 
	We let $\MNoriLoc_{\sigma}(X)$ denote the full subcategory of $\MNori_{\sigma}(X)$ spanned by the motivic local systems:
	it is an abelian subcategory, stable under subquotients and extensions.
\end{nota}

Motivic local systems inherit a canonical monoidal structure, based on the shifted tensor product
\begin{equation*}
	- \otimes^{\dagger} - \colon \MNoriLoc_{\sigma}(X) \times \MNoriLoc_{\sigma}(X) \to \MNoriLoc_{\sigma}(X).
\end{equation*} 

\begin{thm}[{\cite[Thm.~6.3]{Ter24Nori}}]
	The abelian category $\MNoriLoc_{\sigma}(X)$, endowed with the shifted tensor product, is canonically neutral Tannakian over $\Q$.
\end{thm}

Here, the adverb "canonically" refers to the existence of the canonical monoidal exact functor $\iota_{\sigma,X} \colon \MNoriLoc_{\sigma}(X) \to \Loc_{\sigma}(X)$;
the monoidality is inherited from \Cref{thm:Nori-6ff}.
This yields a natural family of fibre functors on motivic local systems:
for every point $x \in X^{\sigma}$, the fibre functor at $x$ is defined as the composite
\begin{equation*}
	\MNoriLoc_{\sigma}(X) \xrightarrow{\iota_{\sigma,X}} \Loc_{\sigma}(X) \xrightarrow{x^{\dagger}} \vect_{\Q},
\end{equation*}
where $x^{\dagger} := x^*[-\dim(X)] \colon \Loc_{\sigma}(X) \to \vect_{\Q}$ denotes the shifted inverse image functor.

\begin{nota}
	For every point $x \in X^{\sigma}$, we let $\Gmot(X,x;\sigma)$ denote the Tannaka dual of $\MNoriLoc_{\sigma}(X)$ with respect to the fibre functor at $x$.
	Similarly, for every $M \in \MNoriLoc_{\sigma}(X)$ we let $\Gmot(M,x;\sigma)$ denote the Tannaka dual of the Tannakian subcategory $\langle M \rangle^{\otimes} \subset \MNoriLoc_{\sigma}(X)$ with respect to the same fibre functor.
\end{nota}

As in the case of Nori motives over fields, a source of non-trivial motivic local systems is the Galois theory of algebraic varieties: 
\begin{ex}[{\cite[Prop.~5.13]{Tub25Artin}}; see also {\cite[Ex.~2.20]{JT25}}]\label[ex]{ex:MNoriLoc_Artin}
	Let $e \colon X' \to X$ be a finite étale cover.
	The associated \textit{Artin motivic local system} is defined as
	\begin{equation*}
		{^p \Q}_{X'/X} := \pi_{\sigma,X}(e_* \unit_{X'}[\dim(X)]) = e_* \pi_{\sigma,X'}(\unit_{X'}[\dim(X)]) \in \MNoriLoc_{\sigma}(X),
	\end{equation*}
	where $\unit_{X'}$ denotes the tensor-unit of $\DAct(X')$.
	If the covering $X'/X$ is Galois, the motivic Galois group $\Gmot({^p \Q}_{X'/X},x;\sigma)$ is canonically isomorphic to the Galois group $\Gal(X'/X)$, regarded as a finite constant algebraic group over $\Q$.
	\qed
\end{ex}

\begin{nota}
	We let $\MNoriLoc_{\sigma}^{\Artin}(X)$ denote the union of the subcategories $\langle {^p \Q}_{X'/X} \rangle^{\otimes}$ as $X'$ varies among all connected finite étale covers of $X$, and we call it the category of \textit{Artin motivic local systems} over $X$:
	it is again a Tannakian subcategory of $\MNoriLoc_{\sigma}(X)$.
\end{nota}

Generalizing the case of Artin motives over a field, the Tannaka dual of $\MNoriLoc_{\sigma}^{\Artin}(X)$ is isomorphic to the étale fundamental group $\pi_1^{\et}(X,x)$, regarded as a pro-finite algebraic group over $\Q$.
The inclusion $\MNoriLoc_{\sigma}^{\Artin}(X) \subset \MNoriLoc_{\sigma}(X)$ determines a canonical surjection
\begin{equation*}
	\Gmot(X,x;\sigma) \twoheadrightarrow \pi_1^{\et}(X,x).
\end{equation*}
Later in the paper, we need to consider the maximal Artin quotients of finite-dimensional motivic Galois groups.
To this end, we introduce the following notation:
\begin{nota}
	For every $M \in \MNoriLoc_{\sigma}(X)$, we let $\langle M \rangle^{\otimes}_{\Artin}$ denote the intersection $\langle M \rangle^{\otimes} \cap \MNori_{\sigma}^{\Artin}(X)$ inside $\MNoriLoc_{\sigma}(X)$:
	it is again a Tannakian subcategory of $\MNoriLoc_{\sigma}(X)$.
	Moreover, we let $\Gmot^{\Artin}(M,x;\sigma)$ denote the Tannaka dual of $\langle M \rangle^{\otimes}_{\Artin}$ with respect to the fibre functor at $x$.
\end{nota}

Under Tannaka duality, the monoidal functor $\iota_{\sigma,k} \colon \MNoriLoc_{\sigma}(X) \to \Loc_{\sigma}(X)$ determines a canonical homomorphism of pro-algebraic groups
\begin{equation}\label{eq:pi1alg-to-Gmot(X)}
	\pi_1^{\alg}(X^{\sigma},x) \to \Gmot(X,x;\sigma).
\end{equation}
Note that the composite homomorphism
\begin{equation*}
	\pi_1^{\alg}(X^{\sigma},x) \to \Gmot(X,x;\sigma) \twoheadrightarrow \pi_1^{\et}(X,x)
\end{equation*}
is still surjective:
it corresponds to the Tannakian subcategory of $\Loc_{\sigma}(X)$ spanned by the local systems arising from finite étale covers of $X$.

Given a morphism between smooth, geometrically connected $k$-varieties $f \colon X \to Y$, the shifted inverse image functor
\begin{equation*}
	f^{\dagger} := f^*[\dim(X) - \dim(Y)] \colon D^b(\MNori_{\sigma}(Y)) \to D^b(\MNori_{\sigma}(X))
\end{equation*}
restricts to a monoidal exact functor
\begin{equation*}
	f^{\dagger} \colon \MNoriLoc_{\sigma}(Y) \to \MNoriLoc_{\sigma}(X).
\end{equation*}
For every base-point $x \in X^{\sigma}$, this determines a homomorphism of Tannaka-dual groups
\begin{equation*}
	\Gmot(X,x;\sigma) \to \Gmot(Y,f(x);\sigma).
\end{equation*}
In the case of the structural projection $a_X \colon X \to \Spec(k)$, we have the following result:

\begin{lem}[{\cite[Rmk.~2.19]{JT25}}]\label[lem]{lem:a_X^*-fullyfaith}
	The shifted inverse image functor
	\begin{equation*}
		a_X^{\dagger} \colon \MNori_{\sigma}(k) \to \MNoriLoc_{\sigma}(X)
	\end{equation*}
	is fully faithful, with essential image stable under subquotients.
\end{lem}
 
This means that the induced homomorphism of Tannaka-dual groups
\begin{equation}\label{eq:Gmot(X)-to-Gmot(k)}
	\Gmot(X,x;\sigma) \to \Gmot(k;\sigma)
\end{equation}
is surjective (by \cite[Prop.~2.21(a)]{DM82}).
If the base-point $x \in X^{\sigma}$ is $k$-rational, this surjection acquires a canonical splitting, induced by the inclusion $x \colon \Spec(k) \hookrightarrow X$.
This suggests that the structure of $\Gmot(X,x;\sigma)$ can be better understood when $x$ is algebraic enough.
However, since a general $k$-variety might not have $k$-rational points at all, it is convenient to look at all $\overline{k}$-rational points.
In this case, the difference between motivic local systems over $X$ and motives over $k$ can be measured as follows: 
\begin{thm}[{\cite[Thm.~7.7(1)]{Jac25}}; {\cite[Thm.~2.24]{JT25}}]\label[thm]{thm:fund_ses}
	For every $x \in X(\overline{k})$, the sequence of pro-algebraic groups
	\begin{equation*}
		\pi_1^{\alg}(X^{\sigma},x) \to \Gmot(X,x;\sigma) \to \Gmot(k;\sigma) \to 1
	\end{equation*}
	defined by the homomorphisms \eqref{eq:pi1alg-to-Gmot(X)} and \eqref{eq:Gmot(X)-to-Gmot(k)} is exact;
	it is even split-exact if $x \in X(k)$.
\end{thm}
The arrow $\pi_1^{\alg}(X^{\sigma},x) \to \Gmot(X,x;\sigma)$ is in general not a closed immersion, because not every local system on $X^{\sigma}$ is a subquotient of one underlying a motivic local system.
The full subcategory of $\Loc_{\sigma}(X)$ spanned by such subquotients is the Tannakian subcategory of \textit{local systems of geometric origin} (see \cite[\S~7.2]{Jac25} or \cite[\S\S~1.3, 2.3]{JT25}).

Before moving on, let us note that the shifted inverse image functor $f^{\dagger} \colon \MNori_{\sigma}(Y) \to \MNori_{\sigma}(X)$ under a morphism $f \colon X \to Y$ restricts to a monoidal functor
\begin{equation*}
	f^{\dagger} \colon \MNoriLoc_{\sigma}^{\Artin}(Y) \to \MNoriLoc_{\sigma}^{\Artin}(X).
\end{equation*}
To see this, let $e_Y \colon Y' \to Y$ be a finite étale morphism, and form the Cartesian square
\begin{equation*}
	\begin{tikzcd}
		X' \arrow{r}{e_X} \arrow{d}{f'} &
		X \arrow{d}{f} \\
		Y' \arrow{r}{e_Y} &
		Y.
	\end{tikzcd}
\end{equation*}
By proper base-change, the canonical morphism in $\MNoriLoc_{\sigma}(X)$
\begin{equation*}
	f^{\dagger} ({^p \Q}_{Y'/Y}) := f^{\dagger} (e_{Y,*} {^p \Q}_{Y'}) \to e_{X,*} (f')^{\dagger} ({^p \Q}_{Y'}) = e_{X,*} {^p \Q}_{X'} =: {^p \Q}_{X'/X}
\end{equation*}
is an isomorphism.
This implies that $f^{\dagger}$ takes $\langle {^p \Q}_{Y'/Y} \rangle^{\otimes}$ to $\langle {^p \Q}_{X'/X} \rangle$, and the conclusion follows by passing to the colimit over the finite étale covers of $Y$.
In particular, for every point $x \in X^{\sigma}$ we have a commutative diagram of the form
\begin{equation*}
	\begin{tikzcd}
		\Gmot(X,x;\sigma) \arrow{r} \arrow[two heads]{d} &
		\Gmot(Y,f(x);\sigma) \arrow[two heads]{d} \\
		\pi_1^{\et}(X,x) \arrow{r} &
		\pi_1^{\et}(Y,f(y)),
	\end{tikzcd}
\end{equation*}
where the lower horizontal arrow is the usual functoriality homomorphism of étale fundamental groups (see \cite[Exp.~V, \S~7]{SGA1}).
The functoriality of motivic local systems and motivic Galois groups, and its restriction to Artin objects, will be frequently used in the course of the next section.

\section{Remarkable relations among motivic Galois groups}\label{sect:Gmot-wow}

The present section contains two technical results comparing motivic Galois groups arising from different motivic categories:
the first one relates motivic local systems over a smooth variety to motives over its function field,
while the second one describes how they change under suitable extensions of the base field.
Their combination plays a crucial role in the proof of our main results in \Cref{sect:MotLocus,sect:WeightLocus}.

As in the previous section, we work over a field $k$ of characteristic $0$ endowed with a complex embedding $\sigma \colon k \hookrightarrow \C$.

\subsection{Generic motivic local systems}

Let $X$ be a smooth, geometrically connected $k$-variety, with function field $k(X)$.
We are going to introduce a generic variant of the category $\MNoriLoc_{\sigma}(X)$:
it is modelled on the category of \textit{generic local systems} $\twocolim_{U \subset X} \Loc_{\sigma}(U)$, where the colimit is taken over all non-empty Zariski opens $U \subset X$.
Note that the canonical functor
\begin{equation*}
	\twocolim_{U \subset X} \Loc_{\sigma}(U) \to \twocolim_{U \subset X} \Perv_{\sigma}(U)
\end{equation*}
is an equivalence, because every perverse sheaf restricts to a shifted local system over some smaller non-empty Zariski open.
 
\begin{nota}
	We let $\MNoriLoc_{\sigma}(k(X))$ denote the universal abelian factorization of the homological functor
	\begin{equation*}
		\beta_{\sigma,k(X)} \colon \twocolim_{U \subset X} \DAct(U) \xrightarrow{(\beta_{\sigma,U})_U} \twocolim_{U \subset X} \Perv_{\sigma}(U) = \twocolim_{U \subset X} \Loc_{\sigma}(U),
	\end{equation*}
	and we call it the category of \textit{generic motivic local systems} over $X$.
\end{nota}

The terminology is justified by the observation that a generic motivic local system can be viewed as a motivic local system defined over some unspecified Zariski-dense open $U \subset X$.
More precisely:
\begin{lem}\label[lem]{lem:MNori(k(X))=colim}
	The canonical functor
	\begin{equation*}
		\twocolim_{U \subset X} \MNoriLoc_{\sigma}(U) \to \twocolim_{U \subset X} \MNori_{\sigma}(U) \to \MNoriLoc_{\sigma}(k(X))
	\end{equation*}
	is an equivalence. 
\end{lem} 
\begin{proof}
	Indeed, both arrows are equivalences:
	the first one by cofinality of local systems inside perverse sheaves (as explained above), the second one by compatibility of universal abelian factorizations with filtered $2$-colimits (see \cite[Lem.~5.12]{Ter24UAF}).
\end{proof}

Note that the abelian category $\twocolim_{U \subset X} \MNoriLoc_{\sigma}(U)$ carries a canonical monoidal structure compatible with that of each $\MNoriLoc_{\sigma}(U)$.
Moreover, the colimit $\twocolim_{U \subset X} \MNoriLoc_{\sigma}^{\Artin}(U)$ can be identified with a full monoidal abelian subcategory of the former, stable under subquotients (since this is the case for each $\MNoriLoc_{\sigma}^{\Artin}(U)$ inside $\MNoriLoc_{\sigma}(U)$).
Via the equivalence of \Cref{lem:MNori(k(X))=colim}, the abelian category $\MNori_{\sigma}(k(X))$ inherits a compatible monoidal structure.

\begin{nota}
	We let $\MNoriLoc_{\sigma}^{\Artin}(k(X))$ denote the full subcategory of $\MNoriLoc_{\sigma}(k(X))$ corresponding to $\twocolim_{U \subset X} \MNoriLoc_{\sigma}^{\Artin}(U)$:
	it is a monoidal abelian subcategory, stable under subquotients.
\end{nota}

The following result identifies motivic local systems over $X$ with those generic local systems which extend to the whole of $X$:
\begin{prop}\label[prop]{prop:MNori_generic}
	The monoidal exact functor
	\begin{equation*}
		\eta_X^{\dagger} \colon \MNoriLoc_{\sigma}(X) \to \twocolim_{U \subset X} \MNoriLoc_{\sigma}(U) \xrightarrow{\sim} \MNori_{\sigma}(k(X))
	\end{equation*} 
	is fully faithful, with essential image stable under subquotients; 
	and similarly for its restriction
	\begin{equation*}
		\eta_X^{\dagger} \colon \MNoriLoc_{\sigma}^{\Artin}(X) \to \twocolim_{U \subset X} \MNoriLoc_{\sigma}^{\Artin}(U) \xrightarrow{\sim} \MNori_{\sigma}^{\Artin}(k(X)).
	\end{equation*}
\end{prop}
\begin{proof}
	In the first place, we establish the properties of $\eta_X^{\dagger} \colon \MNoriLoc_{\sigma}(X) \to \MNori_{\sigma}(k(X))$.
	For this, it suffices to show that the same properties hold for the inverse image functor $j^* \colon \MNoriLoc_{\sigma}(X) \to \MNoriLoc_{\sigma}(U)$ under a Zariski-dense affine open immersion $j \colon U \hookrightarrow X$.
	This is proved in \cite[Lem.~5.12]{JT25}, but we sketch the argument here for the reader's convenience.
	As a consequence of \cite[Lem.~4.3.2]{BBD82}, the intermediate extension functor
	\begin{equation*}
		j_{!*} \colon \MNori_{\sigma}(U) \to \MNori_{\sigma}(X)
	\end{equation*}
	is fully faithful, and its restriction to $\MNoriLoc_{\sigma}(U)$ provides a left-inverse to $j^* \colon \MNoriLoc_{\sigma}(X) \to \MNoriLoc_{\sigma}(U)$.
	This implies formally that the latter is fully faithful.
	To show that its essential image is stable under subquotients, it suffices to show that $j_{!*}$ provides a two-sided inverse to $j^*$ on subobject lattices of motivic local systems.
	This follows from the fact that $j_{!*}$ is left-inverse to $j^*$ on motivic local systems, while $j^*$ is even left-inverse to $j_{!*}$ on motivic perverse sheaves.
	
	In order to show that the desired properties of $\eta_X^{\dagger} \colon \MNoriLoc_{\sigma}(X) \to \MNori_{\sigma}(k(X))$ pass to its restriction to Artin motivic local systems, it suffices to observe that the same properties hold for the inclusions $\MNoriLoc_{\sigma}^{\Artin}(X) \subset \MNoriLoc_{\sigma}(X)$ and $\MNoriLoc_{\sigma}^{\Artin}(k(X)) \subset \MNoriLoc_{\sigma}(k(X))$ (see \Cref{ex:MNoriLoc_Artin}). 
\end{proof}

With some care, the result can be translated in Tannaka-dual terms:
\begin{cor}\label[cor]{cor:Gmot_generic}
	Keep the notation and assumptions of \Cref{prop:MNori_generic}.
	Let $z \in X^{\sigma}$ be a point which is dense for the Zariski topology of $X$.
	Then, we have a canonical commutative diagram of motivic Galois groups
	\begin{equation*}
		\begin{tikzcd}
			\Gmot(\eta_X^{\dagger} M,z;\sigma) \arrow{r}{\sim} \arrow[two heads]{d} &
			\Gmot(M,z; \sigma) \arrow[two heads]{d} \\
			\Gmot^{\Artin}(\eta_X^{\dagger} M,z; \sigma) \arrow{r}{\sim} &
			\Gmot^{\Artin}(M,z; \sigma),
		\end{tikzcd}
	\end{equation*}
	in which both horizontal arrows are isomorphisms.
\end{cor}
\begin{proof}
	The Zariski-denseness assumption means that the point $z$ lies in $U^{\sigma}$ for each Zariski open $U \subset X$, and therefore defines a fibre functor on $\MNoriLoc_{\sigma}(k(X))$.
	Therefore, \Cref{prop:MNori_generic} yields an equivalence of neutralized Tannakian categories $\langle M \rangle^{\otimes} \xrightarrow{\sim} \langle \eta_X^{\dagger} M \rangle^{\otimes}$, which restricts to an analogous equivalence $\langle M \rangle^{\otimes}_{\Artin} \xrightarrow{\sim} \langle \eta_X^{\dagger} M \rangle^{\otimes}_{\Artin}$.
\end{proof}

\begin{rem}
	Of course, the existence of a Zariski-dense point $z \in X^{\sigma}$ depends on $k$:
	for example, when $k = \C$, no such point exists as soon as $\dim(X) \geq 1$
	(see \Cref{constr:Sigma-z} below for a more thorough discussion). 
	In any case, the monoidal abelian category $\MNoriLoc_{\sigma}(k(X))$ is always neutral Tannakian over $\Q$:
	a fibre functor can be obtained by composing the forgetful functor
	\begin{equation*}
		\iota_{\sigma,k(X)} \colon \MNoriLoc_{\sigma}(k(X)) \to \twocolim_{U \subset X} \Loc_{\sigma}(U)
	\end{equation*}
	with a fibre functor for $\twocolim_{U \subset X} \Loc_{\sigma}(U)$, which in turn can be constructed as in \cite[\S~4.4]{BT25}.
	However, we do not need to pursue this approach in the present paper.
\end{rem}

%\begin{cor}\label{cor:Gmot_generic}
%	For every $M \in \MNoriLoc(X)$, the canonical closed immersion
%	\begin{equation*}
%		\Gmot(\eta_X^{\dagger} M;\xi) \hookrightarrow \Gmot(M,x;\sigma)
%	\end{equation*}
%	is an isomorphism, and the induced homomorphism
%	\begin{equation*}
%		\Gmot^{\Artin}(\eta_X^{\dagger} M;\xi) \to \Gmot^{\Artin}(M,x;\sigma)
%	\end{equation*}
%	is an isomorphism as well.
%\end{cor}

\subsection{Nori motives over function fields}

As before, let $X$ be a smooth, geometrically connected $k$-variety.
We want to make sense of $\MNoriLoc_{\sigma}(k(X))$ as a category of Nori motives over the function field $k(X)$.
This is the content of our first auxiliary result:
\begin{prop}\label[prop]{prop:MNori_sigma=MNori_Sigma}
	Assume that there exists a complex embedding $\Sigma \colon k(X) \hookrightarrow \C$ extending $\sigma$.
	Then, there is a canonical monoidal equivalence
	\begin{equation*}
		\MNoriLoc_{\sigma}(k(X)) = \MNori_{\Sigma}(k(X)),
	\end{equation*} 
	which restricts to a monoidal equivalence
	\begin{equation*}
		\MNoriLoc^{\Artin}_{\sigma}(k(X)) = \MNori^{\Artin}_{\Sigma}(k(X)).
	\end{equation*}
\end{prop}

The proof strategy is based on the following geometric idea, the relevance of which in the motivic setting we learned from the proof of \cite[Prop.~2.20]{Ayo14H2}:
\begin{constr}\label[constr]{constr:Sigma-z}
	Keep the notation and assumptions of \Cref{prop:MNori_sigma=MNori_Sigma}.
	The complex embedding $\Sigma \colon k(X) \hookrightarrow \C$ corresponds canonically to a point $z_{\Sigma} \in X^{\sigma} = (X \times_{k,\sigma} \C)(\C)$:
	namely, the morphism $\Spec(\C) \to X \times_{k,\sigma} \C$ induced by the sheaf monomorphism
	\begin{equation*}
		\mathcal{O}_X \subset k(X) \xrightarrow{\Sigma} \C. 
	\end{equation*}
	By construction, the point $z_{\Sigma}$ does not lie on any subset of the form $Y^{\sigma}$ with $Y$ a strict closed algebraic subvariety of $X$:
	in other words, it is dense for the Zariski topology of $X$.
	
	In fact, the rule $\Sigma \mapsto z_{\Sigma}$ defines a natural bijection between complex embeddings of $k(X)$ extending $\sigma$ and points of $X^{\sigma}$ which are dense for the Zariski topology of $X$:
	the inverse rule $z \mapsto \Sigma_z$ associates to a Zariski-dense point $z \in X^{\sigma}$ the complex embedding $\Sigma_z \colon k(X) \hookrightarrow \C$ extending the sheaf monomorphism
	\begin{equation*}
		\mathcal{O}_X \subset \mathcal{O}_{X \times_{k,\sigma} \C} \xrightarrow{z^*} \C. 
	\end{equation*}
	More generally, for every irreducible closed subvariety $S \subset X$, complex embeddings of $k(S)$ extending $\sigma$ correspond to points $t \in X^{\sigma}$ with Zariski-closure $S$.
	\qed 
\end{constr}

\begin{proof}[Proof of \Cref{prop:MNori_sigma=MNori_Sigma}]
	In the first place, we explain how to identify $\MNoriLoc_{\sigma}(k(X))$ and $\MNori_{\Sigma}(k(X))$ as abelian categories.
	Since both of them are defined as universal abelian factorizations, it suffices to show that the two exact functors on the abelian hull
	\begin{equation*}
		\beta_{\sigma,k(X)}^+ \colon \AbHull(\DAct(k(X))) \to \twocolim_{U \subset X} \Loc_{\sigma}(U), \qquad \beta_{\Sigma,k(X)}^+ \colon \AbHull(\DAct(k(X))) \to \vect_{\Q}
	\end{equation*}
	have the same kernel.
	This is proved in \cite[Prop.~1.14]{Ter24Nori}, but we sketch the argument here for the reader's convenience.
	Let $z_{\Sigma}$ denote the Zariski-dense point defined by $\Sigma$ (as in \Cref{constr:Sigma-z}). 
	Taking the stalk at $z_{\Sigma}$ defines a compatible system of fibre functors
	\begin{equation*}
		z_{\Sigma}^{\dagger} \colon \Loc_{\sigma}(U) \to \vect_{\Q}
	\end{equation*} 
	for all Zariski-dense opens $U \subset X$, whence an analogous fibre functor on $\twocolim_{U \subset X} \Loc_{\sigma}(U)$. 
	Unwinding the definitions, one checks that the diagram
	\begin{equation*}
		\begin{tikzcd}
			\twocolim_{U \subset X} \DAct(U) \arrow{rr}{\sim} \arrow{d}{(\beta_{\sigma,U})_U} &&
			\DAct(k(X)) \arrow{rr}{[-\dim(X)]} &&
			\DAct(k(X)) \arrow{d}{\beta_{\Sigma,k(X)}} \\
			\twocolim_{U \subset X} \Perv_{\sigma}(U) &&
			\twocolim_{U \subset X} \Loc_{\sigma}(U) \arrow[equal]{ll} \arrow{rr}{z_{\Sigma}^{\dagger}} &&
			\vect_{\Q}
		\end{tikzcd}
	\end{equation*}
	commutes.
	The sought-after equality of kernels follows formally from the conservativity of $z_{\Sigma}^{\dagger}$.
	
	In order to promote the equivalence of abelian categories $\MNoriLoc_{\sigma}(k(X)) = \MNori_{\Sigma}(k(X))$ obtained in this way to a monoidal equivalence, it suffices to repeat the same argument after replacing $\DAct(U)$ (resp. $\DAct(k(X))$) by its full additive subcategory $(\Bti_{\sigma,U}^*)^{-1}(\Perv_{\sigma}(U))$ (resp. $(\Bti_{\Sigma,k(X)}^*)^{-1}(\vect_{\Q})$).
	Indeed, in view of the way the monoidal structure on Nori motivic perverse sheaves is constructed (see \cite[\S~2]{Ter24Nori}), one easily sees that our equivalence is compatible with the external tensor product.
	Moreover, it is also compatible with inverse images under diagonal immersions, in view of the way these are defined (see \cite[\S~4]{IM24}).
	This implies its compatibility with the internal tensor product.
	
	To deduce the equivalence between the subcategories of Artin objects, it only remains to show that the essential image of $\MNoriLoc_{\sigma}^{\Artin}(k(X))$ inside $\MNori_{\Sigma}(k(X))$ coincides with $\MNori_{\Sigma}^{\Artin}(k(X))$.
	To this end, fix a Zariski-dense open $U \subset X$ together with a finite étale cover $e \colon U' \to U$, and form the Cartesian square of schemes
	\begin{equation*}
		\begin{tikzcd}
			V \arrow{r} \arrow{d}{\tilde{e}} &
			U' \arrow{d}{e} \\
			\Spec(k(X)) \arrow{r}{\eta_U} &
			U.
		\end{tikzcd}
	\end{equation*}
	In view of the canonical isomorphism in $\DAct(k(X))$
	\begin{equation*}
		\eta_U^* e_* \unit_{U'} = \tilde{e}_* \unit_{V},
	\end{equation*}
	the object $\eta_U^{\dagger}({^p \Q}_{U'/U}) \in \MNoriLoc_{\sigma}(k(X))$ corresponds to the object ${^p \Q}_{V/k(X)} \in \MNori_{\Sigma}(k(X))$ under the equivalence of the previous paragraph.
	This implies that the essential image of $\MNoriLoc_{\sigma}^{\Artin}(k(X))$ in $\MNori_{\Sigma}(k(X))$ is contained inside $\MNori_{\Sigma}^{\Artin}(k(X))$.
	Conversely, since every finite étale $k(X)$-algebra is of the form $\mathcal{O}(V)$ for $V$ as above, for a suitable Zariski-dense open $U \subset X$ and a suitable finite étale cover $e \colon U' \to U$, the whole of $\MNori_{\Sigma}^{\Artin}(k(X))$ is contained in the essential image.
\end{proof}

%\begin{rem}
%	There is an alternative way to prove \Cref{prop:MNori_sigma=MNori_Sigma}, by comparing both categories in the statement to the category of Nori motives over $k(X)$ defined via the absolute $\ell$-adic realization (see \cite[Prop. 6.11]{IM24}).
%\end{rem}

We deduce our first technical result about motivic Galois groups:
\begin{cor}\label[cor]{cor:Gmot_sigma=Sigma}
	Keep the notation and assumptions of \Cref{prop:MNori_sigma=MNori_Sigma}.
	Given $M \in \MNoriLoc_{\sigma}(X)$, let $N \in \MNori_{\Sigma}(k(X))$ denote the object corresponding to $\eta_X^{\dagger} M \in \MNoriLoc_{\sigma}(k(X))$.
	Then, there is a canonical commutative diagram of motivic Galois groups
	\begin{equation*}
		\begin{tikzcd}
			\Gmot(N; \Sigma) \arrow{r}{\sim} \arrow[two heads]{d} &
			\Gmot(M,z_{\Sigma}; \sigma) \arrow[two heads]{d} \\
			\Gmot^{\Artin}(N; \Sigma) \arrow{r}{\sim} &
			\Gmot^{\Artin}(M,z_{\Sigma}; \sigma),
		\end{tikzcd}
	\end{equation*}
	in which both horizontal arrows are isomorphisms.
\end{cor}
\begin{proof}
	By construction, the equivalence of \Cref{prop:MNori_sigma=MNori_Sigma} renders the diagram of monoidal functors
	\begin{equation*}
		\begin{tikzcd}
			\MNoriLoc_{\sigma}(k(X)) \arrow[equal]{rr} \arrow{d}{\iota_{\sigma,k(X)}} &&
			\MNori_{\Sigma}(k(X)) \arrow{d}{\iota_{\Sigma,k(X)}} \\
			\twocolim_{U \subset X} \Loc_{\sigma}(U) \arrow{rr}{z_{\Sigma}^{\dagger}} &&
			\vect_{\Q}
		\end{tikzcd}
	\end{equation*}
	commutative.
	Therefore, it determines an equivalence of neutralized Tannakian categories $\langle M \rangle^{\otimes} \xrightarrow{\sim} \langle N \rangle^{\otimes}$, which restricts to an analogous equivalence $\langle M \rangle^{\otimes}_{\Artin} \xrightarrow{\sim} \langle N \rangle^{\otimes}_{\Artin}$.
\end{proof}

\subsection{Topological invariance}\label{subsect:top-inv}

Let now $k'/k$ be a field extension, and suppose that we are given a complex embedding $\sigma' \colon k' \hookrightarrow \C$ extending $\sigma$.
Fix a smooth, geometrically connected $k$-variety $X$, and set $X' := X \times_k k'$.
Our next goal is to compare the categories $\MNoriLoc_{\sigma}(X)$ and $\MNoriLoc_{\sigma'}(X')$.
The starting point is the canonical isomorphism of complex varieties
\begin{equation*}
	X \times_{k,\sigma} \C = (X \times_k k') \times_{k',\sigma'} \C = X' \times_{k',\sigma'} \C.
\end{equation*}
The induced identification of complex-analytic spaces $X^{\sigma} = (X')^{\sigma'}$ allows us to regard $D^b_{\sigma}(X)$ as a full stable sub-$\infty$-category of $D^b_{\sigma'}(X')$.
As $X$ varies, these inclusion functors commute with the six operations, which implies that they are compatible with the perverse $t$-structures.
We thus get a fully faithful exact functor
\begin{equation*}
	\Perv_{\sigma}(X) \to \Perv_{\sigma'}(X'),
\end{equation*} 
which restricts to a monoidal equivalence
\begin{equation*}
	\Loc_{\sigma}(X) \xrightarrow{\sim} \Loc_{\sigma'}(X').
\end{equation*}
For every point $x \in X^{\sigma}$, corresponding to $x' \in (X')^{\sigma'}$, we get an identification of Tannaka-dual groups
\begin{equation*}
	\pi_1^{\alg}((X')^{\sigma'},x') \xrightarrow{\sim} \pi_1^{\alg}(X^{\sigma},x).
\end{equation*}
We want to study the analogous picture in the motivic setting.
In view of the definition of the Betti realization functors, the diagram
\begin{equation*}
	\begin{tikzcd}
		\DAct(X) \arrow{rr} \arrow{d}{\beta_{\sigma,X}} &&
		\DAct(X') \arrow{d}{\beta_{\sigma',X'}} \\
		\Perv_{\sigma}(X) \arrow{rr} &&
		\Perv_{\sigma'}(X')
	\end{tikzcd}
\end{equation*}
commutes up to canonical natural isomorphism (for example, see \cite[Lem.~1.6]{JT25}).
This determines a canonical base-change functor
\begin{equation*}
	\MNori_{\sigma}(X) \to \MNori_{\sigma'}(X')
\end{equation*}
compatible with the forgetful functor as in \eqref{dia:MNori_k'/k}.
By restriction to motivic local systems, we get the commutative diagram of monoidal exact functors
\begin{equation*}
	\begin{tikzcd}
		\MNoriLoc_{\sigma}(X) \arrow{rr} \arrow{d}{\iota_{\sigma,X}} &&
		\MNoriLoc_{\sigma'}(X') \arrow{d}{\iota_{\sigma',X'}} \\
		\Loc_{\sigma}(X) \arrow{rr}{\sim} &&
		\Loc_{\sigma'}(X'),
	\end{tikzcd}
\end{equation*}
which is compatible with the fibre functors at any point $x \in X^{\sigma}$.
Note that the motivic base-change functor is necessarily faithful.
However, in general it might fail to be full, already when $X = \Spec(k)$.

\begin{ex}\label[ex]{ex:non-fullness}
	If $k'/k$ is a non-trivial finite extension, the Artin motive $\Q_{k'/k} \in \MNori_{\sigma}(k)$ discussed in \Cref{ex:Artin_k} is not trivial, but its image in $\MNori_{\sigma'}(k')$ becomes trivial (by construction).
	\qed
\end{ex}

In fact, the only possible cause for the lack of fullness of the base-change functor is precisely the presence of non-trivial Artin motives over $k$ which become trivial over $k'$ --
that is, of non-trivial finite extensions of $k$ inside $k'$.
This is explained in our second auxiliary result: 

\begin{prop}\label[prop]{prop:MNoriLoc_inv}
	Suppose that the field extension $k'/k$ is geometrically integral.
	Then, the monoidal functor
	\begin{equation*}
		\MNoriLoc_{\sigma}(X) \to \MNoriLoc_{\sigma'}(X')
	\end{equation*}
	is fully faithful, with essential image stable under subquotients;
	and similarly for its restriction
	\begin{equation*}
		\MNoriLoc_{\sigma}^{\Artin}(X) \to \MNoriLoc_{\sigma'}^{\Artin}(X').
	\end{equation*}
\end{prop}
\begin{proof}
	Since Artin motivic local systems are stable under subquotients inside all motivic local systems (being a Tannakian subcategory), the desired properties of the entire base-change functor $\MNoriLoc_{\sigma}(X) \to \MNoriLoc_{\sigma'}(X')$ imply the same properties for its restriction to Artin objects.
	So let us prove the conclusion for the entire base-change functor.
	
	To begin with, we reduce ourselves to the case when $X = \Spec(k)$.
	To this end, choose a base-point $x \in X^{\sigma}$, corresponding to $x' \in (X')^{\sigma'}$.
	The thesis holds for $X$ if and only if the homomorphism $\Gmot(X',x';\sigma') \to \Gmot(X,x;\sigma)$ is surjective (by \cite[Prop.~2.21(a)]{DM82}).
	Write $a_X \colon X \to \Spec(k)$ and $a'_{X'} \colon X' \to \Spec(k')$ for the structural morphisms.
	By construction, the diagram of base-change functors
	\begin{equation*}
		\begin{tikzcd}
			\MNori_{\sigma}(k) \arrow{rr} \arrow{d}{a_X^{\dagger}} &&
			\MNori_{\sigma'}(k') \arrow{d}{(a'_{X'})^{\dagger}} \\
			\MNoriLoc_{\sigma}(X) \arrow{rr} &&
			\MNoriLoc_{\sigma'}(X')
		\end{tikzcd}
	\end{equation*}
	commutes up to monoidal natural isomorphism;
	recall also that the base-change functor commutes with the forgetful functors towards local systems.
	Under Tannaka duality, we get the commutative diagram of pro-algebraic groups
	\begin{equation*}
		\begin{tikzcd}
			\pi_1^{\alg}(X',x';\sigma') \arrow{r} \arrow[equal]{d} &
			\Gmot(X',x';\sigma') \arrow{r} \arrow{d} &
			\Gmot(k';\sigma') \arrow{r} \arrow{d} &
			1 \\
			\pi_1^{\alg}(X,x;\sigma') \arrow{r} &
			\Gmot(X,x;\sigma) \arrow{r} &
			\Gmot(k;\sigma) \arrow{r} &
			1,
		\end{tikzcd}
	\end{equation*}
	in which both rows are exact (by \Cref{thm:fund_ses}).
	If the thesis holds for $\Spec(k)$, the right-most arrow in the diagram is surjective (by \cite[Prop.~2.21(a)]{DM82}).
	But then, the surjectivity of the middle vertical arrow follows formally by diagram-chasing, and consequently the thesis holds for $X$ as well (again by \cite[Prop.~2.21(a)]{DM82}).
	
	From now on, assume that $X = \Spec(k)$ and hence $X' = \Spec(k')$. 
	We further reduce ourselves to the case when the extension $k'/k$ is finitely generated (and still geometrically integral).
	To this end, let $\Ical_{k'/k}$ denote the filtered poset of all intermediate fields $k \subset L \subset k'$ which are finitely generated over $k$;
	note that each such extension $L/k$ is still geometrically integral (since so is the ambient extension $k'/k$).
	By \Cref{lem:MNori_colim}, the canonical functor
	\begin{equation*}
		\twocolim_{L \in \Ical_{k'/k}} \MNori_{\sigma'|_L}(L) \to \MNori_{\sigma'}(k')
	\end{equation*}
	is an equivalence.
	Suppose that the thesis is known to hold for all field extensions $L/k$ with $L \in \Ical_{k'/k}$.
	Then, using the explicit description of objects and morphisms in a filtered $2$-colimit, we deduce that the thesis also holds for the functor
	\begin{equation*}
		\MNori_{\sigma}(k) \to \twocolim_{L \in \Ical_{k'/k}} \MNori_{\sigma'|_L}(L),
	\end{equation*}
	which settles the case of the original extension $k'/k$.
	
	Finally, assume that the extension $k'/k$ is finitely generated (and still geometrically integral). 
	Fix a smooth, connected $k$-variety $S$ with function field $k'$, and note that it is automatically geometrically connected (precisely because the extension $k'/k$ is assumed geometrically integral).
	By construction, the base-change functor in the statement factors as
	\begin{equation*}
		\MNori_{\sigma}(k) \xrightarrow{a_S^{\dagger}} \MNoriLoc_{\sigma}(S) \xrightarrow{\eta_S^{\dagger}} \MNoriLoc_{\sigma}(k') = \MNori_{\sigma'}(k')
	\end{equation*}
	where the last passage witnesses the equivalence of \Cref{prop:MNori_sigma=MNori_Sigma}.
	Since both the first and the second arrow are known to be fully faithful, with essential image stable under subquotients (by \Cref{lem:a_X^*-fullyfaith} and \Cref{prop:MNori_generic}, respectively), the same holds for the composite functor.
\end{proof}

\begin{rem}\label[rem]{rem:MNoriLoc_inv+}
	The conclusion of \Cref{prop:MNoriLoc_inv} holds, more generally, for the base-change functor on motivic perverse sheaves over any $k$-variety:
	this can be deduced by Noetherian induction starting from the case of motivic local systems, by adapting the method of \cite[Prop.~1.9]{Ter24Nori}.
\end{rem}

\begin{rem}
	If $k'/k$ is an extension of algebraically closed fields, the base-change functor on Artin motivic local systems is in fact an equivalence.
	This follows essentially from the fact that every finite étale cover of $X'$ is the base-change of a finite étale cover of $X$ (see \cite[Prop.~0BQN]{Stacks}).
\end{rem}

To close the discussion, we deduce our second technical result about motivic Galois groups:
\begin{cor}\label[cor]{cor:Gmot_inv}
	Keep the notation and assumptions of \Cref{prop:MNoriLoc_inv}.
	Given $M \in \MNoriLoc(X)$, let $M' \in \MNoriLoc(X')$ denote its image under the base-change functor.
	Then, for every point $x \in X^{\sigma}$, corresponding to $x' \in (X')^{\sigma'}$, there is a canonical commutative diagram of  motivic Galois groups
	\begin{equation*}
		\begin{tikzcd}
			\Gmot(M',x';\sigma') \arrow{r}{\sim} \arrow[two heads]{d} &
			\Gmot(M,x;\sigma) \arrow[two heads]{d} \\
			\Gmot^{\Artin}(M',x';\sigma') \arrow{r}{\sim} &
			\Gmot^{\Artin}(M,x;\sigma),
		\end{tikzcd}
	\end{equation*}
	in which both horizontal arrows are isomorphisms.
\end{cor}
\begin{proof}
	Indeed, by \Cref{prop:MNoriLoc_inv}, the base-change functor determines an equivalence of neutralized Tannakian categories $\langle M \rangle^{\otimes} \xrightarrow{\sim} \langle M' \rangle^{\otimes}$, which restricts to an analogous equivalence $\langle M \rangle^{\otimes}_{\Artin} \xrightarrow{\sim} \langle M' \rangle^{\otimes}_{\Artin}$.
\end{proof}

%\begin{prop}\label{prop:MArtinLoc_inv}
%	Let $k'/k$ be an extension of algebraically closed fields.
%	Then, the base-change functor on Artin motivic local systems
%	\begin{equation*}
%		\MNoriLoc_{\Artin}(X) \to \MNoriLoc_{\Artin}(X')
%	\end{equation*}
%	is in fact an equivalence.
%\end{prop}
%\begin{proof}
%	As a consequence of \Cref{prop:MNoriLoc_inv}, we know that the functor in the statement is fully faithful, with essential image stable under subquotients.
%	Therefore, it suffices to show that the essential image of $\MNoriLoc(X)$ contains a generating family of the Tannakian category $\MNoriLoc(X')$.
%	We claim that, for every finite étale morphism $e' \colon Y' \to X'$, the object $\Q_{Y'/X'} := e_* \Q_{Y'} \in \MNoriLoc_{\Artin}(X')$ lies in the essential image of $\MNoriLoc(X)$.
%	The key point is that, by \cite[Prop. 58.8.4]{Stacks}, the morphism $e'$ is the base-change of a finite étale morphism $e \colon Y \to X$.
%	By proper base-change, the object $\Q_{Y'/Y}$ is isomorphic to the base-change of the corresponding object $\Q_{Y/X} \in \MNoriLoc_{\Artin}(X)$. 
%\end{proof}

%\begin{cor}\label{cor:Gmot_inv}
%	Keep the assumptions and notation of \Cref{prop:MArtinLoc_inv}.
%	Then, the canonical closed immersion
%	\begin{equation*}
%		\Gmot(M',x') \hookrightarrow \Gmot(M,x)
%	\end{equation*}
%	is an isomorphism, and the induced homomorphism
%	\begin{equation*}
%		\Gmot^{\Artin}(M',x') \to \Gmot^{\Artin}(M,x)
%	\end{equation*}
%	is an isomorphism as well.
%\end{cor}

\section{The exceptional locus}\label{sect:MotLocus}

The present section contains the main results of the paper: 
we define the exceptional locus of a motivic local systems and study its general geometric and arithmetic properties.

We work over varieties defined over the complex numbers or subfields thereof;
the complex embedding used to define Nori motivic local systems and their motivic Galois groups will always be the inclusion, and we omit it from the notation.
 
\subsection{Definition of the exceptional locus}

Let $X$ be a smooth, connected complex variety.
For every point $x \in X(\C)$, the shifted inverse image functor $x^{\dagger} \colon \MNoriLoc(X) \to \MNori(\C)$ determines a closed immersion between Tannaka-dual groups
\begin{equation*}
	\Gmot(\C) \hookrightarrow \Gmot(X,x),
\end{equation*}
which is a section to the canonical homomorphism $\Gmot(X,x) \twoheadrightarrow \Gmot(\C)$ induced by the structural morphism $a_X \colon X \to \Spec(\C)$.
Similarly, for every motivic local system $M \in \MNoriLoc(X)$, the functor $x^{\dagger} \colon \langle M \rangle^{\otimes} \to \langle x^{\dagger} M \rangle^{\otimes}$ determines a closed immersion
\begin{equation*}
	\Gmot(x^{\dagger} M) \hookrightarrow \Gmot(M,x).
\end{equation*}
In particular, we have $\dim(\Gmot(x^{\dagger} M)) \leq \dim(\Gmot(M,x))$, with equality if and only if the inclusion has finite index.

As $x \in X(\C)$ varies, the assignment $x \mapsto \Gmot(M,x)$ can be thought of as a local system of algebraic groups.
It is natural to wonder how far is the assignment $x \mapsto \Gmot(x^{\dagger} M)$ from defining a sub-local system.
In order to understand this, it is convenient to start by removing the contribution of Artin motivic local systems.

\begin{nota}
	For every point $x \in X(\C)$, we set $\Gmot^{\odot}(M,x) := \ker\left\{\Gmot(M,x) \twoheadrightarrow \Gmot^{\Artin}(M,x)\right\}$.
\end{nota}

As $x \in X(\C)$ varies, the assignment $x \mapsto \Gmot^{\Artin}(M,x)$ can be thought of as a quotient local system of $x \mapsto \Gmot(M,x)$.
Therefore, the assignment $x \mapsto \Gmot^{\odot}(M,x)$ defines a sub-local system of it.
Since $\Gmot^{\Artin}(M,x)$ is a finite algebraic group (see \Cref{rem:Gmot^Artin-finite}), the group $\Gmot^{\odot}(M,x)$ has finite index inside $\Gmot(M,x)$.
Hence, its pre-image under the closed immersion $\pi_1^{\alg}(\iota_X(M),x) \hookrightarrow \Gmot(M,x)$ is a normal subgroup of finite index inside the algebraic monodromy group $\pi_1^{\alg}(\iota_X(M),x)$.
In particular, $\Gmot^{\odot}(M,x)$ contains the image of some normal subgroup of finite index of the topological fundamental group $\pi_1(X,x)$.
More precisely:
\begin{lem}
	Let $e \colon X' \to X$ be the finite Galois covering with Galois group $\Gmot^{\Artin}(M,x)$, and let $x' \in X'(\C)$ be a pre-image of $x$.
	Then, the composite homomorphism
	\begin{equation*}
		\pi_1^{\alg}((X')^{\sigma},x') \to \pi_1^{\alg}(X^{\sigma},x) \twoheadrightarrow \pi_1^{\alg}(\iota_X(M),x) \hookrightarrow \Gmot(M,x) 
	\end{equation*}
	factors through $\Gmot^{\odot}(M,x)$.
\end{lem}
\begin{proof}
	The homomorphism in the statement equals the composite
	\begin{equation*}
		\pi_1^{\alg}((X')^{\sigma},x') \twoheadrightarrow \pi_1^{\alg}(e^* \iota_X(M),x') \hookrightarrow \pi_1^{\alg}(\iota_X(M),x) \hookrightarrow \Gmot(M,x),
	\end{equation*}
	which in turn equals the composite
	\begin{equation*}
		\pi_1^{\alg}((X')^{\sigma},x') \twoheadrightarrow \pi_1^{\alg}(\iota_{X'}(e^* M),x') \hookrightarrow \Gmot(e^* M,x') \hookrightarrow \Gmot(M,x).
	\end{equation*}
	The latter factors through $\Gmot^{\odot}(M,x)$, because the composite
	\begin{equation*}
		\Gmot(e^* M,x') \hookrightarrow \Gmot(M,x) \twoheadrightarrow \Gmot^{\Artin}(M,x)
	\end{equation*}
	is already trivial (by construction).
\end{proof}

The group $\Gmot^{\odot}(M,x)$ provides a refined upper bound for $\Gmot(x^{\dagger} M)$, as the following observation shows:
\begin{lem}\label[lem]{lem:Gmot_fibre-to-Artin}
	For every point $x \in X(\C)$, the closed immersion $\Gmot(x^{\dagger} M) \hookrightarrow \Gmot(M,x)$ factors through $\Gmot^{\odot}(M,x)$.
\end{lem}
\begin{proof}
	We have to show that the composite homomorphism
	\begin{equation*}
		\Gmot(x^{\dagger} M) \hookrightarrow \Gmot(M,x) \twoheadrightarrow \Gmot^{\Artin}(M,x)
	\end{equation*} 
	is trivial.
	Under Tannaka duality, this amounts to showing that, for every given $N \in \langle M \rangle^{\otimes}_{\Artin}$, the motive $x^{\dagger} N \in \MNori(\C)$ is trivial.
	To this end, let $e \colon X' \to X$ be the finite Galois covering with Galois group $\Gmot^{\Artin}(M,x)$, and fix a pre-image $x'$ of $x$ in $X'(\C)$.
	We have an isomorphism of motivic Galois groups
	\begin{equation*}
		\Gmot((x')^{\dagger} e^* N) \xrightarrow{\sim} \Gmot(x^{\dagger} N),
	\end{equation*}
	induced by the isomorphism of motives $(x')^{\dagger} e^* N = x^{\dagger} N$.
	But the group on the left-hand side is trivial, because the object $e^* N \in \MNoriLoc(X')$ is already trivial (by construction).
\end{proof}

We are thus led to compare the assignment $x \mapsto \Gmot(x^{\dagger} M)$ to the local system $x \mapsto \Gmot^{\odot}(M,x)$.
The key new definition of the present paper is the following:
\begin{defn}
	We define the \textit{exceptional locus} of $M \in \MNoriLoc(X)$ as the subset $\MotLocus(M) \subset X(\C)$ consisting of all points $x \in X(\C)$ such that the closed immersion $\Gmot(x^{\dagger} M) \hookrightarrow \Gmot^{\odot}(M,x)$ is not an isomorphism.
\end{defn}

Recall that the group $\Gmot(x^{\dagger} M)$ is conjectured to be connected (see \Cref{rem:Gmot(kbar)=connected}).
Hence, the homomorphism $\Gmot(x^{\dagger} M) \hookrightarrow \Gmot(M,x)$ is expected to factor through the identity component $\Gmot^0(M,x)$.
On the other hand, the finite quotient $\Gmot^{\Artin}(M,x)$ is conjectured to be exactly the group of connected components of $\Gmot(M,x)$.
In other words, the subgroup $\Gmot^{\odot}(M,x)$ is conjectured to coincide with $\Gmot^0(M,x)$. 
The conclusion of \Cref{lem:Gmot_fibre-to-Artin} is coherent with this expectation.

\subsection{Geometric description}

As in the previous subsection, let $X$ be a smooth, connected complex variety, and fix a motivic local system $M \in \MNoriLoc(X)$.
We are going to describe the exceptional locus $\MotLocus(M)$ geometrically, showing in particular that it forms a small subset of $X(\C)$ in a precise sense;
this will justify why points in $\MotLocus(M)$ should be considered "exceptional".

As a quick sanity check, note that when $X = \Spec(\C)$ we trivially have $\MotLocus(M) = \emptyset$, because $\Gmot^{\Artin}(\C) = 1$ (as $\C$ is algebraically closed).
Moreover, we always have $\MotLocus(M) = \emptyset$ whenever $M$ is an Artin motivic local system (as a consequence of \Cref{lem:Gmot_fibre-to-Artin} above).

In general, the exceptional locus is a non-trivial and very interesting subset of $X(\C)$.
Its general geometric properties can be summarized as follows:
\begin{thm}\label[thm]{thm:MotLocus}
	The locus $\MotLocus(M) \subset X(\C)$ is a countable union of subsets of the form $S(\C)$, with $S$ a strict closed algebraic subvariety of $X$.
	Moreover, if $X$ admits a model $X_k$ over an algebraically closed subfield $k \subset \C$ such that $M$ belongs to the essential image of $\MNoriLoc(X_k)$, then each maximal closed subvariety inside $\MotLocus(M)$ is defined over $k$ as well.
\end{thm}

This theorem represents the main result of the present paper.
It is a motivic analogue of Cattani--Deligne--Kaplan's celebrated result for the (tensorial) Hodge locus of pure variations of Hodge structure \cite[Thm.~1.1]{CDK95}, extended to admissible mixed variations by Brosnan--Pearlstein--Schnell in \cite[Cor.~2]{BPS10}.

In the course of the proof, we will use an auxiliary observation of independent interest:
\begin{lem}\label[lem]{lem:MotLocus_birat}
	Let $p \colon X' \to X$ be a birational morphism between smooth, connected complex varieties.
	Then, for every $M \in \MNoriLoc(X)$, we have the equality
	\begin{equation*}
		\MotLocus(p^* M) = p^{-1}(\MotLocus(M)).
	\end{equation*}
\end{lem}
\begin{proof}
	Note that the inverse image functor
	\begin{equation*}
		p^{\dagger} = p^* \colon \MNoriLoc(X) \to \MNoriLoc(X')
	\end{equation*}
	is fully faithful, with essential image stable under subquotients:
	indeed, it fits into the commutative diagram
	\begin{equation*}
		\begin{tikzcd}
			\MNoriLoc(X) \arrow{rr}{p^*} \arrow{d} &&
			\MNoriLoc(X') \arrow{d} \\
			\MNoriLoc(\C(X)) \arrow{rr}{\sim} &&
			\MNoriLoc(\C(X')),
		\end{tikzcd}
	\end{equation*}
	in which both vertical arrows enjoy the same properties (by \Cref{prop:MNori_generic}) and the lower horizontal arrow is even an equivalence (as $p$ is birational, by hypothesis).
	Since Artin motivic local systems form Tannakian subcategories, we see that the same properties hold for the restriction 
	\begin{equation*}
		p^{\dagger} = p^* \colon \MNoriLoc^{\Artin}(X) \to \MNoriLoc^{\Artin}(X').
	\end{equation*}
	As a consequence, for every point $x' \in X'(\C)$ we have a commutative diagram of the form
	\begin{equation*}
		\begin{tikzcd}
			\Gmot(p^* M,x') \arrow{r}{\sim} \arrow[two heads]{d} &
			\Gmot(M,p(x')) \arrow[two heads]{d} \\
			\Gmot^{\Artin}(p^* M,x') \arrow{r}{\sim} &
			\Gmot^{\Artin}(M,p(x')),
		\end{tikzcd}
	\end{equation*}
	whence an isomorphism
	\begin{equation*}
		\Gmot^{\odot}(p^* M,x') \xrightarrow{\sim} \Gmot^{\odot}(M,p(x')).
	\end{equation*}
	The latter fits into the commutative diagram
	\begin{equation*}
		\begin{tikzcd}
			\Gmot((x')^{\dagger} p^* M) \arrow[hook]{r} \arrow{d}{\sim} &
			\Gmot^{\odot}(p^* M, x') \arrow{d}{\sim} \\
			\Gmot(p(x')^{\dagger} M) \arrow[hook]{r} &
			\Gmot^{\odot}(M,p(x')),
		\end{tikzcd}
	\end{equation*}
	in which the left-most vertical arrow is induced by the isomorphism of motives $(x')^{\dagger} p^* M = p(x')^{\dagger} M$.
	Thus, the upper horizontal arrow is an isomorphism if and only if so is the lower horizontal arrow.
\end{proof}

\begin{rem}
	The result remain true if $p \colon X' \to X$ is a dominant morphism such that the induced function field extension $\C(X')/\C(X)$ is geometrically integral.
	Indeed, one can show that the conclusion of \Cref{prop:MNoriLoc_inv} still holds for the base-change functor $\MNoriLoc(\C(X)) \to \MNoriLoc(\C(X'))$.
	%One way to obtain this is by applying the usual version of \Cref{prop:MNoriLoc_inv} with respect to a choice of compatible complex embeddings of $\C(X)$ and $\C(X')$.
	%This is legitimate, because the categories of Nori motives and their functoriality ultimately depend only on the scheme and not on similar auxiliary choices (by \cite[Prop.~6.11]{IM24} and \cite[Rmk.~2.7(1), Rmk.~3.5(1)]{Ter24UAF}).
\end{rem}

\begin{proof}[Proof of \Cref{thm:MotLocus}]
	Fix a countable, algebraically closed subfield $k \subset \C$ such that $X$ admits a model $X_k$ over $k$ and $M$ is isomorphic to the image of an object $M_k \in \MNoriLoc(X_k)$ under the base-change functor $\MNoriLoc(X_k) \to \MNoriLoc(X)$:
	the existence of such a field is guaranteed by \Cref{lem:MNori(X)_colim}, applied to the filtered poset of countable subfields of $\C$;
	the object $M_k$ is uniquely determined up to canonical natural isomorphism (as a consequence of \Cref{prop:MNoriLoc_inv}).
	From now on, we tacitly identify $X$ with $X_k \times_k \C$, and similarly for the associated complex-analytic spaces (as in the discussion at the beginning of \Cref{subsect:top-inv}).
	Applying \Cref{cor:Gmot_inv} to the field extension $\C/k$, for every $x \in X(\C)$ we get a canonical isomorphism
	\begin{equation*}
		\Gmot(M,x) \xrightarrow{\sim} \Gmot(M_k,x).
	\end{equation*}
	We have to show that $\MotLocus(M)$ is a union of subsets of the form $S(\C)$, where $S = S_k \times_k \C$ with $S_k$ a strict closed subvariety of $X_k$;
	if this is the case, it is automatically a countable union (because $k$ is countable, by assumption).
	
	In the first place, we claim that every point $z \in X(\C)$ which is dense for the Zariski topology of $X_k$ lies outside $\MotLocus(M)$.
	To this end, let $\Sigma_z \colon k(X_k) \hookrightarrow \C$ denote the corresponding complex embedding (as in \Cref{constr:Sigma-z}), and choose a complex embedding $\overline{\Sigma}_z \colon \overline{k(X_k)} \hookrightarrow \C$ extending $\Sigma_z$.
	By construction, we have a commutative diagram of monoidal exact functors
	\begin{equation*}
		\begin{tikzcd}
			\MNoriLoc(X_k) \arrow{r} \arrow{d}{\eta_{X_k}^{\dagger}} & 
			\MNoriLoc(X) \arrow{r}{z^{\dagger}} & 
			\MNori(\C) \\
			\MNoriLoc(k(X_k)) \arrow[equal]{r} &
			\MNori_{\Sigma_z}(k(X_k)) \arrow{r} &
			\MNori_{\overline{\Sigma}_z}(\overline{k(X_k)}), \arrow{u}
		\end{tikzcd}
	\end{equation*}
	compatible with the fibre functor at $z$.
	Let $N \in \MNori_{\Sigma_z}(k(X_k))$ denote the object corresponding to the generic motivic local system $\eta_{X_k}^{\dagger} M_k \in \MNoriLoc(k(X_k))$ under \Cref{prop:MNori_sigma=MNori_Sigma}, and write $\overline{N}$ for its image in $\MNori_{\overline{\Sigma}_z}(\overline{k(X_k)})$.
	Under Tannaka duality, we obtain the commutative diagram of motivic Galois groups
	\begin{equation*}
		\begin{tikzcd}
			\Gmot(M_k,z) &
			\Gmot(M,z) \arrow{l}{\sim} &
			\Gmot(z^{\dagger} M) \arrow{l} \arrow{d}{\sim} \\
			\Gmot(\eta_{X_k}^{\dagger} M_k, z) \arrow{u}{\sim} &
			\Gmot(N; \Sigma_z) \arrow[equal]{l} &
			\Gmot(\overline{N}; \overline{\Sigma}_z). \arrow{l}
		\end{tikzcd}
	\end{equation*}
	Here, the right-most vertical arrow is an isomorphism by \Cref{cor:Gmot_inv} applied to the field extension $\C/\overline{k(X_k)}$, and the three isomorphisms in the left part of the diagram (given by \Cref{cor:Gmot_inv}, \Cref{cor:Gmot_generic}, and \Cref{cor:Gmot_sigma=Sigma}) respect the Artin quotients.
	As a consequence, the closed immersion $\Gmot(z^{\dagger} M) \hookrightarrow \Gmot^{\odot}(M,z)$ is an isomorphism if and only if the sequence
	\begin{equation*}
		1 \to \Gmot(\overline{N};\overline{\Sigma}_z) \to \Gmot(N;\Sigma_z) \to \Gmot^{\Artin}(N;\Sigma_z) \to 1
	\end{equation*}
	is exact, which is indeed the case (by \Cref{cor:ses_Gmot-Artin}).
	This proves the claim.
	
	It follows that $\MotLocus(M)$ is contained inside the union of all subsets of $X(\C)$ of the form $S(\C)$, where $S = S_k \times_k \C$ with $S_k$ a strict closed subvariety of $X_k$.  
	It remains to show that $\MotLocus(M)$ equals the union of some of these subsets $S(\C)$.
	To this end, fix an irreducible, strict closed subvariety $S_k \subset X_k$, set $S := S_k \times_k \C$, and choose a point $t \in S(\C)$ which is dense for the Zariski topology of $S_k$.
	We are going to show that $S(\C) \subset \MotLocus(M)$ as soon as $t \in \MotLocus(M)$;
	since every point of $X(\C)$ is Zariski-dense inside a unique irreducible closed subvariety of $X_k$, this will conclude the proof.
	
	Suppose first that $S_k$ is smooth over $k$, and let $i \colon S \hookrightarrow X$ denote the associated closed immersion.
	Then, the closed immersion $\Gmot(t^{\dagger} M) \hookrightarrow \Gmot^{\odot}(M,t)$ factors as
	\begin{equation*}
		\Gmot(t^{\dagger} M) \hookrightarrow \Gmot^{\odot}(i^{\dagger} M,t) \hookrightarrow \Gmot^{\odot}(M,t).
	\end{equation*}
	But the first arrow is already known to be an isomorphism (thanks to the above claim, applied to the object $i^{\dagger} M \in \MNoriLoc(S)$ with respect to the point $t$).
	Hence, we have $t \in \MotLocus(M)$ if and only if the closed immersion
	\begin{equation*}
		\Gmot^{\odot}(i^{\dagger} M,t) \hookrightarrow \Gmot^{\odot}(M,t) 
	\end{equation*}
	is not an isomorphism.
	If this is the case, the same conclusion holds for all point $s \in S(\C)$, which in turn implies that all such points belong to $\MotLocus(M)$, as wanted.
	
	If $S_k$ is not smooth, let $Z_k$ denote its singular locus.
	By Hironaka's work on Embedded Resolution of Singularities \cite{Hir64}, there exists a proper birational morphism $p \colon X'_k \to X_k$ from a smooth $k$-variety $X'_k$ such that the Zariski-closure $S'_k$ of $p^{-1}(S_k \setminus Z_k)$ inside $X'_k$ is smooth and the induced morphism $S'_k \to S_k$ restricts to an isomorphism over $S_k \setminus Z_k$.
	As usual, set $Z := Z_k \times_k \C$, $X' := X'_k \times_k \C$, and $S' := S'_k \times_k \C$.
	Note that the point $t$ lies outside $Z(\C)$ (being a generic point of $S_k$), and so it has a unique pre-image $t' \in S'(\C)$, which is automatically dense for the Zariski topology of $S'_k$ (by construction).
	Then, we have the chain of equivalences
	\begin{equation*}
		t \in \MotLocus(M) \iff t' \in \MotLocus(p^* M) \iff S'(\C) \subset \MotLocus(p^* M) \iff S(\C) \subset \MotLocus(M),
	\end{equation*}
	where the first and last passages follow from \Cref{lem:MotLocus_birat} while the middle passage follows from the smooth case.
	This concludes the proof.
\end{proof}

\begin{rem}\label[rem]{rem:MotLocus_closure}
	As a consequence of the theorem, the locus $\MotLocus(M)$ can never contain a non-empty analytic open subset of $X(\C)$.
	The proof also shows that, if $S$ is an irreducible closed subvariety of $X$ such that $\MotLocus(M) \supset U(\C)$ for some non-empty Zariski open $U \subset S$, then $\MotLocus(M) \supset S(\C)$.
	This can be also deduced directly from the fact that $\MotLocus(M)$ is a countable union of closed subvarieties.
\end{rem}

Recall that, if $M$ is an Artin motivic local system, then for each point $x \in X(\C)$ the motive $x^{\dagger} M \in \MNori(\C)$ is trivial (being an Artin motive).
As a consequence of our main result, the converse holds as well:
\begin{cor}\label[cor]{cor:Artin=trivial_stalk}
	A motivic local system $M \in \MNoriLoc(X)$ is Artin as soon as the motive $x^{\dagger} M \in \MNori(\C)$ is trivial for all $x \in X(\C)$.
\end{cor}
\begin{proof}
	Note that $M$ is Artin if and only if $\Gmot^{\odot}(M,x) = 1$ for some (or equivalently, for all) $x \in X(\C)$.
	Hence, if $M$ is not Artin, the motive $x^{\dagger} M$ is not trivial whenever $x \notin \MotLocus(M)$.
	Since $\MotLocus(M) \neq X(\C)$ (by \Cref{thm:MotLocus}), the conclusion follows.  
\end{proof}

%The structural theorem 
%\begin{rem}
%	The main theorem has a 
%	Let $p \colon X' \to X$ be a dominant morphism between smooth connected complex varieties such that the function field extension $\C(X')/\C(X)$ is geometrically integral.
%	Then, for every $M \in \MNoriLoc(X)$, the functor
%	\begin{equation*}
%		\langle M \rangle^{\otimes}_{\Artin} \to \langle f^* M \rangle^{\otimes}_{\Artin}
%	\end{equation*}
%\end{rem}

%As a consequence of our main result, 
%In fact, the following more precise statement holds:
%\begin{cor}\label{cor:MotLocus_closure}
%	If $S$ is an irreducible, strict closed subvariety of $X$ such that $\MotLocus(M) \supset U(\C)$ for some non-empty Zariski open $U \subset S$, then $\MotLocus(M) \supset S(\C)$.
%\end{cor}
%\begin{proof}
%	Since the conclusion trivially holds when $S = \Spec(\C)$, we may assume that $\dim(S) \geq 1$.
%	By \Cref{thm:MotLocus}, we can write
%	\begin{equation*}
%		\MotLocus(M) = \bigcup_{i \in I} S_i(\C)
%	\end{equation*}
%	with $I$ a countable index set and each $S_i$ a strict closed subvariety of $X$.
%	This yields the inclusion
%	\begin{equation*}
%		U(\C) \subset S(\C) \cap \MotLocus(M)  = \bigcup_{i \in I} (S \cap S_i)(\C),
%	\end{equation*}
%	whence the equality
%	\begin{equation*}
%		U(\C) = \bigcup_{i \in I} (U \cap S_i)(\C).
%	\end{equation*}
%	Since a positive-dimensional, irreducible complex variety cannot be contained into a countable union of strict closed subvarieties, we must have $U \subset S_i$ for some $i \in I$.
%	But then, we must have $S \subset S_i$ as well (since $S_i$ is closed).
%\end{proof}

\subsection{Stability under Galois conjugation}\label{subsect:MotLocus-Gal}

Throughout this subsection, we fix a subfield $k \subset \C$, and we let $\overline{k} \subset \C$ denote its algebraic closure.
We work over a smooth, connected complex variety $X$ admitting a model $X_{\overline{k}}$ over $\overline{k}$.

Let $M \in \MNoriLoc(X)$ be a motivic local system which is isomorphic to the image of an object $M_{\overline{k}} \in \MNoriLoc(X_{\overline{k}})$ under the base-change functor $\MNoriLoc(X_{\overline{k}}) \to \MNoriLoc(X)$.
By \Cref{thm:MotLocus}, the exceptional locus $\MotLocus(M)$ is then defined over $\overline{k}$.
From an arithmetic viewpoint, it is natural to wonder what changes when $\overline{k}$ gets replaced by $k$.
This is the content of the following result: 
\begin{prop}\label[prop]{prop:MotLocus-Gal}
	Assume that $X$ admits a model $X_k$ over $k$ such that $M$ belongs to the essential image of $\MNoriLoc(X_k)$.
	Then, the $\Gal(\overline{k}/k)$-action on $X$ induced by the Galois action on $X_{\overline{k}}$ stabilizes the locus $\MotLocus(M)$.
\end{prop}

%Note that, if $k$ is not itself algebraically closed, in general $M$ might have several pairwise non-isomorphic models over $k$ (see \Cref{ex:non-fullness}).
%However, up to canonical isomorphism they all give rise to the same model over $\overline{k}$ (as a consequence of \Cref{prop:MNoriLoc_inv}).
%Intuitively, the proposition follows from the $\Gal(\overline{k}/k)$-invariance of the latter object.

Before getting to the proof, we need to comment on the $\Gal(\overline{k}/k)$-action on motivic local systems defined over $k$.

\begin{constr}\label[constr]{constr:f^*-Gal}
	Given a $k$-variety $X_k$, set $X_{\overline{k}} := X_k \times_k \overline{k}$, and let $e \colon X_{\overline{k}} \to X_k$ denote the base-change morphism.
	Moreover, fix an element $\varphi \in \Gal(\overline{k}/k)$, let $f \colon X_{\overline{k}} \xrightarrow{\sim} X_{\overline{k}}$ denote the corresponding $k$-automorphism. 
	We construct a canonical monoidal exact functor
	\begin{equation*}
		f^* \colon \MNoriLoc(X_{\overline{k}}) \to \MNoriLoc(X_{\overline{k}}),
	\end{equation*}
	and we show that it restricts to
	\begin{equation*}
		f^* \colon \MNoriLoc^{\Artin}(X_{\overline{k}}) \to \MNoriLoc^{\Artin}(X_{\overline{k}}).
	\end{equation*}
	We want to obtain it as the restriction of an exact functor
	\begin{equation*}
		f^* \colon \MNori(X_{\overline{k}}) \to \MNori(X_{\overline{k}})
	\end{equation*}
	rendering the diagram
	\begin{equation*}
		\begin{tikzcd}
			\DAct(X_{\overline{k}}) \arrow{rr}{f^*} \arrow{d} &&
			\DAct(X_{\overline{k}}) \arrow{d} \\
			\MNori(X_{\overline{k}}) \arrow{rr}{f^*} &&
			\MNori(X_{\overline{k}})
		\end{tikzcd}
	\end{equation*}
	strictly commutative, so we start by constructing the latter.
	If such a functor exists, it is uniquely determined by this commutativity property (by \cite[Rmk.~2.7(1)]{Ter24UAF}): 
	this implies that it is a self-equivalence, with canonical quasi-inverse $(f^{-1})^*$;
	the same then holds for its restriction to motivic local systems.
	
	By \cite[Prop.~2.5]{Ter24UAF}, the sought-after exact functor exists if and only if the self-equivalence of the abelian hull $f^* \colon \AbHull(\DAct(X_{\overline{k}})) \xrightarrow{\sim} \AbHull(\DAct(X_{\overline{k}}))$ stabilizes the kernel of the exact functor $\beta_{\overline{\sigma},X_{\overline{k}}}^+ \colon \AbHull(\DAct(X_{\overline{k}})) \to \Perv_{\overline{\sigma}}(X_{\overline{k}})$.
	Since this coincides with the kernel of the $\ell$-adic exact functor $\beta_{\ell,X_{\overline{k}}}^+ \colon \AbHull(\DAct(X_{\overline{k}})) \to \Perv_{\ell}(X_{\overline{k}})$ (see the proof of \cite[Prop.~6.11]{IM24}), it suffices to prove the stability of the latter.
	But this follows from the existence of the commutative diagram
	\begin{equation*}
		\begin{tikzcd}
			\DAct(X_{\overline{k}}) \arrow{rr}{f^*} \arrow{d}{\Real_{\ell,X_{\overline{k}}}} &&
			\DAct(X_{\overline{k}}) \arrow{d}{\Real_{\ell,X_{\overline{k}}}} \\
			D^b_{\ell}(X_{\overline{k}}) \arrow{rr}{f^*} \arrow{d}{{^p H^0}} &&
			D^b_{\ell}(X_{\overline{k}}) \arrow{d}{{^p H^0}} \\
			\Perv_{\ell}(X_{\overline{k}}) \arrow{rr}{f^*} &&
			\Perv_{\ell}(X_{\overline{k}})
		\end{tikzcd}
	\end{equation*} 
	witnessing the compatibility of the $\ell$-adic realization with inverse image functors (see \cite[Thm.~6.9]{Ayo14Et}).
	This defines the sought-after exact functor on motivic perverse sheaves.
	Since this functor renders the diagram
	\begin{equation*}
		\begin{tikzcd}
			\MNori(X_{\overline{k}}) \arrow{rr}{f^*} \arrow{d}{\iota_{\ell,X}} &&
			\MNori(X_{\overline{k}}) \arrow{d}{\iota_{\ell,X}} \\
			\Perv_{\ell}(X_{\overline{k}}) \arrow{rr}{f^*} &&
			\Perv_{\ell}(X_{\overline{k}})
		\end{tikzcd}
	\end{equation*}
	commutative up to natural isomorphism, it restricts to an exact functor between the subcategories of motivic local systems.
	By examining the construction of the monoidal structure on Nori motivic sheaves (see \cite[\S\S~2, 4]{Ter24Nori}), one sees that the latter is canonically monoidal, as wanted.
	Moreover, for every finite étale cover $W \to X_{\overline{k}}$, we have a canonical isomorphism in $\MNoriLoc(X_{\overline{k}})$
	\begin{equation*}
		f^* \Q_{W/X_{\overline{k}}} \xrightarrow{\sim} \Q_{W/X_{\overline{k}}}.
	\end{equation*}
	This gives the desired restriction to Artin motivic local systems.
	
	Now set $X := X_k \times_k \C = X_{\overline{k}} \times_{\overline{k}} \C$, and write $f \colon X \to X$ also for the induced scheme automorphism (which is not a $\C$-automorphism unless $\varphi = 1$). 
	Arguing as above, we obtain a canonical monoidal exact functor
	\begin{equation*}
		f^* \colon \MNoriLoc(X) \to \MNoriLoc(X)
	\end{equation*}
	which restricts to
	\begin{equation*}
		f^* \colon \MNoriLoc^{\Artin}(X) \to \MNoriLoc^{\Artin}(X),
	\end{equation*}
	compatibly with the previous ones under the base-change functor $\MNoriLoc(X_{\overline{k}}) \to \MNoriLoc(X)$.
	Clearly, the whole construction is compatible with restriction to Zariski opens of $X_{\overline{k}}$.
    \qed
\end{constr}

\begin{lem}\label[lem]{lem:f^*-stalks}
	Keep the notation and assumptions of \Cref{constr:f^*-Gal}.
	Then, for every point $x \in X(\C)$, the diagram of monoidal functors
	\begin{equation*}
		\begin{tikzcd}
			\MNoriLoc(X) \arrow{rr}{f^*} \arrow{dr}{f(x)^{\dagger}} &&
			\MNoriLoc(X) \arrow{dl}{x^{\dagger}} \\
			&
			\MNori(\C)
		\end{tikzcd}
	\end{equation*}
	commutes up to monoidal natural isomorphism.
\end{lem}
\begin{proof}
	We prove that, more generally, the diagram of monoidal triangulated functors
	\begin{equation*}
		\begin{tikzcd}
			D^b(\MNori(X)) \arrow{rr}{f^*} \arrow{dr}{f(x)^*} &&
			D^b(\MNori(X)) \arrow{dl}{x^*} \\
			& 
			D^b(\MNori(\C))
		\end{tikzcd}
	\end{equation*}
	commutes up to monoidal natural isomorphism.
	
	To begin with, let us ignore monoidality and construct a natural isomorphism filling this diagram.
	This is equivalent to giving a natural isomorphism filling the diagram of right-adjoint functors
	\begin{equation*}
		\begin{tikzcd}
			D^b(\MNori(X)) &&
			D^b(\MNori(X)) \arrow{ll}{f_*} \\
			&
			D^b(\MNori(\C)). \arrow{ul}{f(x)_*} \arrow{ur}{x_*}
		\end{tikzcd}
	\end{equation*}
	Since this diagram, in turn, is obtained from the diagram of exact functor
	\begin{equation*}
		\begin{tikzcd}
			\MNori(X) &&
			\MNori(X) \arrow{ll}{f_*} \\
			&
			\MNori(\C), \arrow{ul}{f(x)_*} \arrow{ur}{x_*}
		\end{tikzcd}
	\end{equation*}
	it suffices to give a natural isomorphism filling the latter.
	Since all the exact functors in question come from the lifting principles of universal abelian factorizations (see \cite[\S\S~2,3]{Ter24UAF}), such a natural isomorphism is induced by the analogous natural isomorphism on Voevodsky motivic sheaves.
	
	To conclude, we need to promote the natural isomorphism just obtained to a monoidal one.
	To this end, observe that the above construction is compatible with the external tensor product on Nori motivice perverse sheaves (as defined in \cite[\S~2]{Ter24Nori}):
	this essentially follows from the external monoidality of direct image functors under closed immersions.
	Therefore, in order to show that the construction is compatible with the internal tensor product, it suffices to check its compatibility with diagonal closed immersions.
	This essentially comes down to constructing a natural isomorphism filling the diagram of triangulated functors
	\begin{equation*}
		\begin{tikzcd}
			D^b(\MNori(X \times X)) \arrow{rr}{(f \times f)^*} \arrow{dr}{\Delta_X^*} &&
			D^b(\MNori(X \times X)) \arrow{dl}{\Delta_X^*} \\
			&
			D^b(\MNori(X)),
		\end{tikzcd}
	\end{equation*}
	where $\Delta_X \colon X \hookrightarrow X \times X$ denotes the diagonal embedding.
	And this can be done in the same way as for the previous natural isomorphism.
\end{proof}

\begin{cor}\label[cor]{cor:f^*-stalks}
	Keep to notation and assumptions of \Cref{lem:f^*-stalks}, and fix $M \in \MNoriLoc(X)$.
	Then, for every $x \in X(\C)$, we have a commutative diagram of motivic Galois groups
		\begin{equation*}
		\begin{tikzcd}
			\Gmot(x^{\dagger} f^* M) \arrow[hook]{r} \arrow{d}{\sim} &
			\Gmot(f^* M,x) \arrow{d}{\sim} \arrow[two heads]{r} &
			\Gmot^{\Artin}(f^* M,x) \arrow{d}{\sim} \\
			\Gmot(f(x)^{\dagger} M) \arrow[hook]{r} &
			\Gmot(M, f(x)) \arrow[two heads]{r} &
			\Gmot^{\Artin}(M, f(x)),
		\end{tikzcd}
	\end{equation*}
	in which all vertical arrows are isomorphisms.
\end{cor}
\begin{proof}
	Indeed, by \Cref{lem:f^*-stalks}, the functor $f^*$ induces an equivalence of Tannakian categories $\langle M \rangle^{\otimes} \xrightarrow{\sim} \langle f^* M \rangle^{\otimes}$, which restricts to an analogous equivalence $\langle M \rangle^{\otimes}_{\Artin} \xrightarrow{\sim} \langle f^* M \rangle^{\otimes}_{\Artin}$ (by \Cref{constr:f^*-Gal}) and intertwines the stalks at $x$ and $f(x)$ as well as the corresponding fibre functors.
\end{proof}

\begin{proof}[Proof of \Cref{prop:MotLocus-Gal}]
	As in \Cref{constr:f^*-Gal}, fix an element $\varphi \in \Gal(\overline{k}/k)$, let $f \colon X_{\overline{k}} \xrightarrow{\sim} X_{\overline{k}}$ denote the corresponding $k$-automorphism, and write also $f \colon X \xrightarrow{\sim} X$ for the induced scheme automorphism.
	We have to prove that the permutation map on complex points $f \colon X(\C) \xrightarrow{\sim} X(\C)$ stabilizes the subset $\MotLocus(M)$.
	
	To this end, fix a point $x \in \MotLocus(M)$, and let us show that $f(x) \in \MotLocus(M)$ as well.
	From \Cref{cor:f^*-stalks}, we get the commutative diagram of motivic Galois groups
	\begin{equation*}
		\begin{tikzcd}
			\Gmot(x^{\dagger} f^* M) \arrow[hook]{r} \arrow{d}{\sim} &
			\Gmot^{\odot}(f^* M,x) \arrow{d}{\sim} \\
			\Gmot(f(x)^{\dagger} M) \arrow[hook]{r} &
			\Gmot^{\odot}(M,f(x)).
		\end{tikzcd}
	\end{equation*}
	In order to conclude the proof, it suffices to exhibit an isomorphism in $\MNoriLoc(X)$
	\begin{equation*}
		f^* M \xrightarrow{\sim} M.
	\end{equation*}
	Indeed, this will give the commutative diagram of motivic Galois groups
	\begin{equation*}
		\begin{tikzcd}
			\Gmot(x^{\dagger} f^* M) \arrow[hook]{r} \arrow{d}{\sim} &
			\Gmot^{\odot}(f^* M, x) \arrow{d}{\sim} \\
			\Gmot(x^{\dagger} M) \arrow{r}{\sim} &
			\Gmot^{\odot}(M,x),
		\end{tikzcd}
	\end{equation*}
	and the thesis will follow by comparing its lower row with the lower row of the previous diagram.
	
	To construct the sought-after isomorphism, we use the hypothesis that $M$ is defined over $k$.
	In detail, fix an object $M_k \in \MNoriLoc(X_k)$ such that $M$ is isomorphic to its image under the base-change functor $\MNoriLoc(X_k) \to \MNoriLoc(X)$;
	the object $M_k$ is not unique up to isomorphism (see \Cref{ex:non-fullness}), but its image in $\MNoriLoc(X_{\overline{k}})$ is so (as a consequence of \Cref{prop:MNoriLoc_inv}).
	Write $e \colon X_{\overline{k}} \to X_k$ for the base-change morphism, and write $e^* \colon \MNoriLoc(X_k) \to \MNoriLoc(X_{\overline{k}})$ for the corresponding base-change functor.
    By the uniqueness property in \Cref{constr:f^*-Gal}, we get a canonical natural monoidal isomorphism between functors $\MNori(X_k) \to \MNori(X_{\overline{k}})$
    \begin{equation*}
	    e^* = (ef)^*,
    \end{equation*}
    induced by the analogous natural isomorphism of functors $\DAct(X_k) \to \DAct(X_{\overline{k}})$.
    This induces a canonical isomorphism in $\MNoriLoc(X_{\overline{k}})$
    \begin{equation*}
    	f^* M_{\overline{k}} = f^* e^* M_k = (ef)^* M_k = e^* M_k = M_{\overline{k}}.
    \end{equation*}
    Since the functors $f^* \colon \MNoriLoc(X_{\overline{k}}) \to \MNoriLoc(X_{\overline{k}})$ and $f^* \colon \MNoriLoc(X) \to \MNoriLoc(X)$ are compatible under the base-change functor $\MNoriLoc(X_{\overline{k}}) \to \MNoriLoc(X)$ (as noted at the end of \Cref{constr:f^*-Gal}), this induces the sought-after isomorphism in $\MNoriLoc(X)$. 
\end{proof}

\section{The weight-splitting locus}\label{sect:WeightLocus}

Following the approach of the previous section, we study the splitting locus of the weight filtration on motivic local systems over complex varieties.
Our main results generalize to the splitting locus of arbitrary short exact sequences of motivic local systems.

\subsection{Recollections on the motivic weight filtration}

We start by reviewing the theory of motivic weights, following mostly \cite[\S~6]{IM24}.
This is modelled on the analogous theory for mixed $\ell$-adic complexes (see \cite{Hub97} and \cite{Mor24}).

To fix notation, let $k$ be a field of characteristic $0$ endowed with a complex embedding $\sigma \colon k \hookrightarrow \C$, and let $X$ be a $k$-variety.
The abelian category $\MNori_{\sigma}(X)$, defined in terms of the Betti realization (as at the beginning of \Cref{subsect:MNori(X)}), admits an alternative presentation in terms of the $\ell$-adic realization of Voevodsky motivic sheaves (by \cite[Prop.~6.11]{IM24}).
This determines a canonical faithful exact functor
\begin{equation*}
	\iota_{\ell,X} \colon \MNori_{\sigma}(X) \rightarrow \Perv_{\ell}^{\textup{mf}}(X)
\end{equation*}
into the abelian category of mixed filtered $\ell$-adic perverse sheaves (see \cite[Prop. 6.17]{IM24}).
Recall that every object $K \in \Perv_{\ell}^{\textup{mf}}(X)$ is endowed with a canonical weight filtration
\begin{equation*}
	0 \subset \dots \subset W_{\leq a-1} K \subset W_{\leq a} K \subset \dots \subset K,
\end{equation*}
uniquely characterized by the property that $W_{\leq a} K$ (resp. $K/W_{\leq a-1} K$) is of weight at most $a$ (resp. at least $a$) for all $a \in \Z$.

\begin{nota}
	For every $a \in \Z$, we let $\MNori_{\sigma}^{w \leq a}(X)$ (resp. $\MNori_{\sigma}^{w \geq a}(X)$) denote the full subcategory of $\MNori_{\sigma}(X)$ spanned by the objects $M$ such that $\iota_{\ell,X}(M)$ is of weight at most $a$ (resp. at least $a$).
	Moreover, we let $\MNori_{\sigma}^{w=a}(X)$ denote the intersection $\MNori_{\sigma}^{w \leq a}(X) \cap \MNori_{\sigma}^{w \geq a}(X)$.
\end{nota}

\begin{prop}[{\cite[\S\S~6.3-6.4]{Ara13}}; {\cite[Prop.~6.17]{IM24}}]
	Every object $M \in \MNori_{\sigma}(X)$ is endowed with a canonical weight filtration
	\begin{equation*}
		0 \subset \dots \subset W_{\leq a-1} M \subset W_{\leq a} M \subset \dots \subset M,
	\end{equation*}
	uniquely characterized by the property that $W_{\leq a} M \in \MNori_{\sigma}^{w \leq a}(X)$ and $M/W_{\leq a-1} M \in \MNori_{\sigma}^{w \geq a}(X)$ for all $a \in \Z$.
	Every morphism in $\MNori_{\sigma}(X)$ is strictly compatible with the weight filtrations.
\end{prop}

The whole difficulty in the proof is to show the existence of some filtration as in the statement for every given object $M$.
Strict compatibility with respect to morphisms 
%$M_1 \to M_2$ means that, for every $a \in \Z$, the diagram
%\begin{equation*}
%	\begin{tikzcd}
%		W_{\leq a} M_1 \arrow{r} \arrow{d} &
%		W_{\leq a} M_2 \arrow{d} \\
%		M_1 \arrow{r} &
%		M_2
%	\end{tikzcd}
%\end{equation*}
%is Cartesian, and
follows from the analogous property for the weight filtration of mixed perverse sheaves.
It implies that there are no non-zero morphisms between pure objects of different weights.
Hence, the subcategories $\MNori_{\sigma}^{w=a}(X)$ are in direct sum inside $\MNori_{\sigma}(X)$ as $a \in \Z$ varies.
Moreover, each $\MNori_{\sigma}^{w=a}$, and hence also their direct sum, is stable under subquotients inside $\MNori_{\sigma}(X)$.
In fact, there is more to say:
\begin{prop}[{\cite[Thm.~6.24]{IM24}}]\label[prop]{prop:pure=semisimple}
	For every $a \in \Z$, the abelian subcategory $\MNori_{\sigma}^{w=a}(X)$ is semi-simple.
\end{prop}

\begin{cor}\label[cor]{cor:semisimple}
	The subcategory of semi-simple objects inside $\MNori_{\sigma}(X)$ equals $\bigoplus_{a \in \Z} \MNori_{\sigma}^{w=a}(X)$.
\end{cor}
\begin{proof}
	Indeed, \Cref{prop:pure=semisimple} implies that an object $M \in \MNori_{\sigma}(X)$ is semi-simple if and only if it splits as the direct sum of its weight-graded pieces.
\end{proof}

\begin{rem}\label[rem]{rem:weights-bc}
	Let $k'/k$ be a finite field extension, and choose a complex embedding $\sigma' \colon k' \hookrightarrow \C$ extending $\sigma$.
	Given a $k$-variety $X$, set $X' := X \times_k k'$.
	Then, the base-change functor $\MNori_{\sigma}(X) \to \MNori_{\sigma'}(X')$ is compatible with the weight filtrations (since the same holds for the underlying functor $\Perv_{\ell}^{\textup{mf}}(X) \to \Perv_{\ell}^{\textup{mf}}(X')$).
\end{rem}

%By \cite[Cor. 6.27(1)]{IM24}, the weight filtration on perverse motives extends to a weight structure on $D^b(\MNori(X))$, in the sense of \cite{Bon14}.
%As a consequence of \cite{BBD82}, for every morphism of $k$-varieties $f \colon X \to Y$, the inverse image functor $f^* \colon D^b(\MNori(Y)) \to D^b(\MNori(X))$ takes $D^b_{w \leq a}(\MNori(Y))$ to $D^b_{w \leq a}(\MNori(X))$ for all $a \in \Z$.

Suppose now that $X$ is smooth and geometrically connected over $k$.
Since the abelian subcategory $\MNoriLoc_{\sigma}(X) \subset \MNori_{\sigma}(X)$ is stable under subquotients, it contains all terms of the weight filtration of its objects, as well as all of their weight-graded pieces.

\begin{nota}
	For every $a \in \Z$, we let $\MNoriLoc_{\sigma}^{w \leq a}(X)$ denote the intersection $\MNoriLoc_{\sigma}(X) \cap \MNori_{\sigma}^{w\leq a}(X)$, and similarly for $\MNoriLoc_{\sigma}^{w \geq a}(X)$ and  $\MNoriLoc_{\sigma}^{w=a}(X)$.
	Moreover, we let $\MNoriLoc_{\sigma}^{\pure}(X)$ denote the intersection $\MNoriLoc_{\sigma}(X) \cap \bigoplus_{a \in \Z} \MNori_{\sigma}^{w=a}(X) = \bigoplus_{a \in \Z} \MNoriLoc_{\sigma}^{w=a}(X)$.
\end{nota}

\begin{ex}\label[ex]{ex:Q_X-w=dim(X)}
	The shifted unit object ${^p \Q}_X \in \MNoriLoc_{\sigma}(X)$ is pure of weight $\dim(X)$:
	this corresponds to the fact that the ordinary unit object $\Q_X \in D^b(\MNori_{\sigma}(X))$ is pure of weight $0$.
	It follows that, for all $a \in \Z$, the duality endofunctor of $\MNoriLoc_{\sigma}(X)$ exchanges $\MNoriLoc_{\sigma}^{w \leq a + \dim(X)}(X)$ with $\MNoriLoc_{\sigma}^{w \geq -a + \dim(X)}(X)$ (as a consequence of \cite[\S~5.1.14(ii*)]{BBD82}).
	\qed
\end{ex}

By general Tannakian formalism, the abelian subcategory of semi-simple objects inside a neutral Tannakian category in characteristic $0$ is in fact a Tannakian subcategory.
In the case of motivic local systems, we have:
\begin{lem}\label[lem]{lem:MNoriLoc-semisimple}
	The Tannakian subcategory of semi-simple objects inside $\MNoriLoc_{\sigma}(X)$ equals $\MNoriLoc_{\sigma}^{\pure}(X)$. 
\end{lem}
\begin{proof}
	Since $\MNoriLoc_{\sigma}(X)$ is stable under subquotients inside $\MNori_{\sigma}(X)$, it follows from \Cref{cor:semisimple} that an object $M \in \MNoriLoc_{\sigma}(X)$ is semi-simple if and only if it belongs to $\MNoriLoc_{\sigma}^{\pure}(X)$.
	
	One can also check directly that $\MNoriLoc_{\sigma}^{\pure}(X)$ is stable under the shifted tensor product on $\MNoriLoc_{\sigma}(X)$.
	Indeed, for all $a, b \in \Z$, the shifted tensor product restricts to
	\begin{equation*}
		- \otimes^{\dagger} - \colon \MNoriLoc_{\sigma}^{w \leq a + \dim(X)}(X) \times \MNoriLoc_{\sigma}^{w \leq b + \dim(X)}(X) \to \MNoriLoc_{\sigma}^{w \leq a+b+\dim(X)}(X)
	\end{equation*}  
	(by \cite[\S~5.1.14(ii)]{BBD82}).
	Dually, it restricts to
	\begin{equation*}
		- \otimes^{\dagger} - \colon \MNoriLoc_{\sigma}^{w \geq a+\dim(X)}(X) \times \MNoriLoc_{\sigma}^{w \geq b+\dim(X)}(X) \to \MNoriLoc_{\sigma}^{w \geq a+b+\dim(X)}(X),
	\end{equation*}
	(see \Cref{ex:Q_X-w=dim(X)}), and therefore to
	\begin{equation*}
		- \otimes^{\dagger} - \colon \MNoriLoc_{\sigma}^{w=a+\dim(X)}(X) \times \MNoriLoc_{\sigma}^{w=b+\dim(X)}(X) \to \MNoriLoc_{\sigma}^{w=a+b+\dim(X)}(X).
	\end{equation*} 
\end{proof}

\begin{nota}
	For every point $x \in X^{\sigma}$, we let $\Gmot^{\pure}(X,x)$ denote the Tannaka dual of $\MNoriLoc_{\sigma}^{\pure}(X)$ with respect to the fibre functor at $x$.
\end{nota}

Under Tannaka duality, the inclusion $\MNoriLoc_{\sigma}^{\pure}(X) \subset \MNoriLoc_{\sigma}(X)$ determines a surjective homomorphism
\begin{equation*}
	\Gmot(X,x;\sigma) \twoheadrightarrow \Gmot^{\pure}(X,x;\sigma),
\end{equation*}
which identifies $\Gmot^{\pure}(X,x;\sigma)$ with the maximal pro-reductive quotient of $\Gmot(X,x;\sigma)$.
An object $M \in \MNoriLoc_{\sigma}(X)$ is semi-simple if and only if the abelian category $\langle M \rangle^{\otimes}$ is semi-simple.
By Tannaka duality, this happens if and only if the group $\Gmot(M,x;\sigma)$ is reductive for some (or, equivalently, for all) $x \in X^{\sigma}$.

Note that, since the formation of pro-reductive quotients is not functorial with respect to arbitrary homomorphisms of pro-algebraic groups, a general monoidal functor between neutral Tannakian categories needs not send semi-simple objects to semi-simple objects.
In the motivic setting, however, we have the following positive result:

\begin{lem}\label[lem]{lem:f^*-pure}
	Let $f \colon X \to Y$ be a morphism between smooth, geometrically connected $k$-varieties.
	Then, the shifted inverse image functor $f^{\dagger} \colon \MNoriLoc_{\sigma}(Y) \to \MNoriLoc_{\sigma}(X)$ restricts to
	\begin{equation*}
		f^{\dagger} \colon \MNoriLoc_{\sigma}^{\pure}(Y) \to \MNoriLoc_{\sigma}^{\pure}(X).
	\end{equation*}
\end{lem}
\begin{proof}
	Indeed, for every $a \in \Z$, the shifted inverse image functor restricts to
	\begin{equation*}
		f^{\dagger} \colon \MNoriLoc_{\sigma}^{w \leq a + \dim(Y)}(Y) \to \MNoriLoc_{\sigma}^{w \leq a + \dim(X)}(X)
	\end{equation*}  
	(by \cite[\S~5.1.14(i)]{BBD82}).
	Dually, it restricts to
	\begin{equation*}
		f^{\dagger} \colon \MNoriLoc_{\sigma}^{w \geq a + \dim(Y)}(Y) \to \MNoriLoc_{\sigma}^{w \geq a + \dim(X)}(X),
	\end{equation*} 
	(see \Cref{ex:Q_X-w=dim(X)}), and therefore to
	\begin{equation*}
		f^{\dagger} \colon \MNoriLoc_{\sigma}^{w=a+\dim(Y)}(Y) \to \MNoriLoc_{\sigma}^{w=a+\dim(X)}(X).
	\end{equation*}
\end{proof}

In particular, for every point $x \in X^{\sigma}$ we have a commutative diagram of the form
\begin{equation*}
	\begin{tikzcd}
		\Gmot(X,x;\sigma) \arrow{r} \arrow[two heads]{d} &
		\Gmot(Y,f(x);\sigma) \arrow[two heads]{d} \\
		\Gmot^{\pure}(X,x;\sigma) \arrow{r} &
		\Gmot^{\pure}(Y,f(x);\sigma).
	\end{tikzcd}
\end{equation*}

\begin{rem}\label[rem]{rem:weights_sigma-Sigma} 
	Suppose that we are given a complex embedding $\Sigma \colon k(X) \hookrightarrow \C$ extending $\sigma$.
	Then, the equivalence $\MNoriLoc_{\sigma}(k(X)) = \MNori_{\Sigma}(k(X))$ of \Cref{prop:MNori_sigma=MNori_Sigma} is compatible with weights (as these are defined in terms of the absolute $\ell$-adic realization).
\end{rem}

\subsection{Definition and main properties}

We want to study the properties of the weight filtration on motivic local systems over complex varieties.
As in \Cref{sect:MotLocus}, we omit any mention of complex embeddings from the notation,
and we identify the complex-analytic space underlying a given complex variety with the set of its complex points.

In the following, we let $X$ be a smooth, connected complex variety. 

\begin{defn}
	We define the \textit{weight-splitting locus} of $M \in \MNoriLoc(X)$ as the subset $\WeightLocus(M) \subset X(\C)$ consisting of all points $x \in X(\C)$ such that the motive $x^{\dagger} M \in \MNori(\C)$ is semi-simple.
\end{defn}

The terminology is justified by the above discussion: 
a point $x \in X(\C)$ lies in $\WeightLocus(M)$ if and only if the weight filtration of $x^{\dagger} M$ splits.
Since semi-simple motives form a Tannakian subcategory (by \Cref{lem:MNoriLoc-semisimple}), the object $x^{\dagger} M$ is semi-simple if and only if the abelian category $\langle x^{\dagger} M\rangle^{\otimes}$ is semi-simple.
By Tannaka duality, this happens if and only if the motivic Galois group $\Gmot(x^{\dagger} M)$ is reductive.
 
Note that, contrarily to the exceptional locus, the weight-splitting locus needs not form a strict subset of $X(\C)$:
indeed, we always have $\WeightLocus(M) = X(\C)$ whenever $M$ is semi-simple (as a consequence of \Cref{lem:f^*-pure}).
However, this is the only case when it happens.
In fact, the basic geometric properties of the weight-splitting locus can be summarized as follows: 
\begin{thm}\label[thm]{thm:WeightLocus}
	Suppose that $M$ is not semi-simple.
	Then, the locus $\WeightLocus(M)$ is a countable union of subsets of $\MotLocus(M)$ the form $S(\C)$, with $S$ a strict closed algebraic subvariety of $X$.
	Moreover, if $X$ admits a model $X_k$ over an algebraically closed subfield $k \subset \C$ such that $M$ belongs to the essential image of $\MNoriLoc(X_k)$, then each maximal closed subvariety inside $\WeightLocus(M)$ is defined over $k$ as well.
\end{thm}

Note the close resemblance to the statement of \Cref{thm:MotLocus}.
The starting point of the proof is the following semi-simplicity criterion, which we are going to apply to motives over function fields:

\begin{lem}\label[lem]{lem:split-bc}
	Let $K$ be a field of characteristic $0$ endowed with a complex embedding $\Sigma \colon K \hookrightarrow \C$;
	let $\overline{K}$ be an algebraic closure of $K$, and fix a complex embedding $\overline{\Sigma} \colon \overline{K} \hookrightarrow \C$ extending $\Sigma$.
	Then, a short exact sequence in $\MNori_{\Sigma}(K)$
	\begin{equation*}
		\Sfrak \colon 0 \to N_1 \to N \to N_2 \to 0
	\end{equation*}
	splits if and only if so does its image in $\MNori_{\overline{\Sigma}}(\overline{K})$
	\begin{equation*}
		\overline{\Sfrak} \colon 0 \to \overline{N}_1 \to \overline{N} \to \overline{N}_2 \to 0.
	\end{equation*}
\end{lem}
\begin{proof}
	It is obvious that $\overline{\Sfrak}$ splits if so does $\Sfrak$, so we only have to show the converse implication.
	To this end, let $K'/K$ be the finite Galois extension with Galois group $\Gmot^{\Artin}(N;\Sigma)$, write $\Sigma' := \overline{\Sigma}|_{K'}$, and let
	\begin{equation*}
		\Sfrak' \colon 0 \to N'_1 \to N' \to N'_2 \to 0
	\end{equation*}
	denote the image of $\Sfrak$ in $\MNori_{\Sigma'}(K')$.
	Since the base-change functor $\langle N' \rangle^{\otimes} \to \langle \overline{N} \rangle^{\otimes}$ is an equivalence (by \Cref{cor:N'_to_Nbar_fullyfaith}), the sequence above splits as soon as $\overline{\Sfrak}$ splits.
	We have to deduce that, in this case, the sequence $\Sfrak$ splits as well.
	To this end, note that the base-change functor $\MNori_{\Sigma}(K) \to \MNori_{\Sigma'}(K')$ factors through an equivalence
	\begin{equation*}
		\MNori_{\Sigma}(K) \xrightarrow{\sim} \MNori_{\Sigma}(\Spec(K'))^{\Gal(K'/K)} = \MNori_{\Sigma'}(K')^{\Gal(K'/K)},
	\end{equation*}
	obtained from \Cref{lem:MNori(X)=MNori(X')^Gal} and \Cref{rem:MNori_sigma(k')=MNori_sigma'(k')}.
	Now fix a section $\phi \colon N'_2 \to N'$ to the epimorphism $N' \twoheadrightarrow N'_2$.
	By construction, the morphism in $\MNori_{\Sigma}(\Spec(K'))$
	\begin{equation*}
		[K' \colon K]^{-1} \sum_{g \in \Gal(K'/K)} g \cdot \phi \colon N'_2 \to N'
	\end{equation*}
	is a $\Gal(K'/K)$-invariant section, which thus descends to a section $N_2 \to N$ to the original epimorphism $N \twoheadrightarrow N_2$.
\end{proof}

\begin{cor}\label[cor]{cor:semisimple-bc}
	Keep the notation and assumptions of \Cref{lem:split-bc}.
	Then, an object $N \in \MNori_{\Sigma}(K)$ is semi-simple if and only if so is its image $\overline{N} \in \MNori_{\overline{\Sigma}}(\overline{K})$.
\end{cor}
\begin{proof}
	Let again $K'/K$ be the finite Galois extension with Galois group $\Gmot^{\Artin}(N;\Sigma)$, write $\Sigma' := \overline{\Sigma}|_{K'}$, and let $N'$ denote the image of $N$ in $\MNori_{\Sigma'}(K')$.
	As in the proof of \Cref{lem:split-bc}, we reduce ourselves to showing that $N$ is semi-simple if and only if so is $N'$.
	This amounts to showing that the weight filtration of $N$ splits if and only if so does the weight filtration of $N'$ (by \Cref{cor:semisimple}).
	We claim that in fact, for every $a \in \Z$, the short exact sequence in $\MNori_{\Sigma}(K)$
	\begin{equation*}
		0 \to W_{\leq a-1} N \to N \to W_{\geq a} N \to 0
	\end{equation*}
	splits if and only if so does the short exact sequence in $\MNori_{\Sigma'}(K')$
	\begin{equation*}
		0 \to W_{\leq a-1} N' \to N' \to W_{\geq a} N' \to 0.
	\end{equation*}
	Since the second sequence is the image of the first one under the base-change functor (see \Cref{rem:weights-bc}), the claim follows from the proof of \Cref{lem:split-bc}.
\end{proof}

An object $M \in \MNoriLoc(X)$ is semi-simple if and only if the Tannaka-dual group $\Gmot(M,x)$ is reductive for some (or, equivalently, for all) $x \in X(\C)$.
The previous results allow us to refine this criterion as follows:
\begin{prop}\label[prop]{prop:M_semisimple=Gmot0_reductive}
	An object $M \in \MNoriLoc(X)$ is semi-simple if and only if the group $\Gmot^{\odot}(M,x)$ is reductive for some (or equivalently, for all) $x \in X(\C)$.
\end{prop}
\begin{proof}
	It suffices to show that, for a well-chosen point $z \in X(\C)$, the group $\Gmot(M,z)$ is reductive if and only if so is its subgroup $\Gmot^{\odot}(M,z)$.
	To choose such a point, we proceed as in the proof of \Cref{thm:MotLocus}: 
	fix a countable, algebraically closed subfield $k \subset \C$ such that $X$ admits a model $X_k$ over $k$ and $M$ is isomorphic to the image of an object $M_k \in \MNoriLoc(X_k)$ under the base-change functor $\MNoriLoc(X_k) \to \MNori(X)$, and let $z \in X(\C)$ be dense for the Zariski topology of $X_k$.
	We prove that the conclusion holds for this point.
	
	To this end, let $\Sigma_z \colon k(X_k) \hookrightarrow \C$ denote the complex embedding corresponding to $z$ (as in \Cref{constr:Sigma-z}), and fix a complex embedding $\overline{\Sigma}_z \colon \overline{k(X_k)} \hookrightarrow \C$ extending it.
	Let $N \in \MNori_{\Sigma_z}(k(X_k))$ denote the object corresponding to $\eta_{X_k}^{\dagger} M_k \in \MNoriLoc(k(X_k))$ under \Cref{prop:MNori_sigma=MNori_Sigma}, and write $\overline{N}$ for its image in $\MNori_{\overline{\Sigma}_z}(\overline{k(X_k)})$.
	As in the proof of \Cref{thm:MotLocus}, we can identify the closed immersion $\Gmot^{\odot}(M,z) \hookrightarrow \Gmot(M,z)$ with the closed immersion $\Gmot(\overline{N};\overline{\Sigma}_z) \hookrightarrow \Gmot(N;\Sigma_z)$.
	Hence, we are left to show that $\Gmot(N;\Sigma_z)$ is reductive if and only if so is $\Gmot(\overline{N};\overline{\Sigma}_z)$.
	Under Tannaka duality, this amounts to saying that $N$ is semi-simple if and only if so is $\overline{N}$, which is indeed the case (by \Cref{cor:semisimple-bc}).
\end{proof}

This result is useful to compare the weight-splitting locus with the exceptional locus.
The proof strategy of \Cref{thm:WeightLocus} is essentially the same as for \Cref{thm:MotLocus}.
Before getting to it, we record one more observation:
\begin{lem}\label[lem]{lem:i^*WeightLocus}
	Let $f \colon X \to Y$ be any morphism between smooth, connected complex varieties.
	Then, for every $M \in \MNoriLoc_{\sigma}(Y)$, we have
	\begin{equation*}
		\WeightLocus(f^{\dagger} M) = f^{-1} \WeightLocus(M).
	\end{equation*}
\end{lem}
\begin{proof}
	Indeed, for every point $x \in X(\C)$, we have a canonical isomorphism of motives $x^{\dagger} f^{\dagger} M = f(x)^{\dagger} M$.
	In particular, the first one is semi-simple if and only if so is the second one.
\end{proof}

\begin{proof}[Proof of \Cref{thm:WeightLocus}]
	As in the proof of \Cref{thm:MotLocus}, fix a countable, algebraically closed subfield $k \subset \C$ such that $X$ admits a model $X_k$ over $k$ and $M$ is isomorphic to the image of an object $M_k \in \MNoriLoc(X_k)$ under the base-change functor $\MNoriLoc(X_k) \to \MNoriLoc_{\sigma}(X)$.
	From now on, we tacitly identify $X$ with $X_k \times_k \C$, as similarly for the associated complex-analytic spaces.
	
	In the first place, we prove the inclusion $\WeightLocus(M) \subset \MotLocus(M)$.
	To this end, fix a point $x \in X(\C) \setminus \MotLocus(M)$.
	By definition, the group $\Gmot(x^{\dagger} M)$ is isomorphic to $\Gmot^{\odot}(M,x)$.
	Since $M$ is not semi-simple (by hypothesis), the latter group is not reductive (by \Cref{prop:M_semisimple=Gmot0_reductive}).
	Under Tannaka duality, this means that $x^{\dagger} M$ is not semi-simple.
	This proves the desired inclusion.
	
	In view of \Cref{thm:MotLocus}, it follows that $\WeightLocus(M)$ is contained inside the union of all subsets of the form $S(\C)$ where $S = S_k \times_k \C$ with $S_k$ a strict closed subvariety of $X_k$ contained in $\MotLocus(M)$.
	It remains to show that $\WeightLocus(M)$ equals the union of some of these subsets $S(\C)$.
	To this end, fix an irreducible, strict closed subvariety $S_k \subset X_k$, set $S := S_k \times_k \C$, and choose a point $t \in S(\C)$ which is dense of the Zariski topology of $S_k$.
	It suffices to show that $S(\C) \subset \WeightLocus(M)$ as soon as $t \in \WeightLocus(S)$.
	
	If $S_k$ is smooth, let $i \colon S \hookrightarrow X$ denote the associated closed immersion.
	Since $i^{-1}(\WeightLocus(M)) = \WeightLocus(i^{\dagger} M)$ (by \Cref{lem:i^*WeightLocus}), it suffices to show that the object $i^{\dagger} M \in \MNoriLoc(S)$ is semi-simple as soon as $t \in \WeightLocus(i^{\dagger} M)$.
	Indeed, if $i^{\dagger} M$ is not semi-simple, then we have $\WeightLocus(i^{\dagger} M) \subset \MotLocus(i^{\dagger} M)$ (by the above claim, applied to $i^{\dagger} M$), but we know that $t \notin \MotLocus(i^{\dagger} M)$ (by \Cref{thm:MotLocus}).
	
	If $S_k$ is not smooth, we use Embedded Resolution of Singularities \cite{Hir64}: 
	this gives us a proper birational morphism $p \colon X'_k \to X_k$ from a smooth variety $X'_k$ endowed with a smooth closed subvariety $S'_k$ such that $p$ restricts to a proper birational morphism $S'_k \to S_k$.
	As usual, set $S' := S'_k \times_k \C$, and let $t' \in S'(\C)$ denote the unique pre-image of $t$.
	We have the chain of equivalences
	\begin{equation*}
		t \in \WeightLocus(M) \iff t' \in \WeightLocus(p^* M) \iff S'(\C) \subset \WeightLocus(p^* M) \iff S(\C) \subset \WeightLocus(M),
	\end{equation*}
	where the first and last passages follow from \Cref{lem:i^*WeightLocus} while the middle passage follows from the smooth case.
	This concludes the proof. 
\end{proof}

\begin{rem}
	As in \Cref{rem:MotLocus_closure} one sees that, if $S$ is an irreducible closed subvariety of $X$ such that $\WeightLocus(M) \supset U(\C)$ for some non-empty Zariski open $U \subset S$, then $\WeightLocus(M) \supset S(\C)$.
\end{rem}

As explained above, if the motivic local system $M$ is semi-simple, then so is the motive $x^{\dagger} M$.
As a consequence of our main results, the converse holds as well:

\begin{cor}\label[cor]{cor:semi-simple_stalk}
	A motivic local system $M \in \MNoriLoc(X)$ is semi-simple as soon as the motive $x^{\dagger} M \in \MNori(\C)$ is semi-simple for all $x \in X(\C)$.
\end{cor}
\begin{proof}
	Indeed, suppose that $M$ is not semi-simple.
	Then, we have $\WeightLocus(M) \subset \MotLocus(M)$ (by \Cref{thm:WeightLocus}).
	Hence, the motive $x^{\dagger} M$ is not semi-simple whenever $x \notin \MotLocus(M)$.
	Since $\MotLocus(M) \neq X(\C)$ (by \Cref{thm:MotLocus}), the conclusion follows.
\end{proof}
Next, we discuss the stability properties of the weight-splitting locus under Galois conjugation.
For this, we restore the notation of \Cref{subsect:MotLocus-Gal}:
we fix a subfield $k \subset \C$ and let $\overline{k} \subset \C$ denote its algebraic closure.
By \Cref{thm:WeightLocus}, if $X$ and $M$ admit models over $\overline{k}$, the weight-splitting locus $\WeightLocus(M)$ is defined over $\overline{k}$ as well.
We now study what happens when these models further descend to $k$. 

\begin{prop}\label[prop]{prop:WeightLocus-Gal}
	Assume that $X$ admits a model $X_k$ over $k$ such that $M$ belongs to the essential image of $\MNoriLoc(X_k)$.
	Then, the $\Gal(\overline{k}/k)$-action on $X$ induced by the Galois action on $X_{\overline{k}}$ stabilizes the locus $\WeightLocus(M)$.
\end{prop}
\begin{proof}
	We proceed as for \Cref{prop:MotLocus-Gal}.
	Given an element $\varphi \in \Gal(\overline{k}/k)$, let $f \colon X_{\overline{k}} \xrightarrow{\sim} X_{\overline{k}}$ denote the corresponding $k$-automorphism, and write also $f \colon X \xrightarrow{\sim} X$ for the induced scheme automorphism.
	We have to prove that the permutation map on complex points $f \colon X(\C) \xrightarrow{\sim} X(\C)$ stabilizes the subset $\WeightLocus(M)$.
	
	So fix a point $x \in \WeightLocus(M)$, and let us show that $f(x) \in \WeightLocus(M)$ as well.
	To this end, it suffices to recall that we have an isomorphism of motives
	\begin{equation*}
		f(x)^{\dagger} M = x^{\dagger} f^* M = x^{\dagger} M,
	\end{equation*}
	where the second passage follows as in the proof of \Cref{prop:MotLocus-Gal}.
	In particular, since $x^{\dagger} M$ is assumed to be semi-simple, so must be $f(x)^{\dagger} M$.
\end{proof}

\subsection{Variant for short exact sequences}

There is an alternative approach to studying the weight-splitting locus of motivic local systems, which leads to slightly more precise results.
It is based on the following notion:

%Saying that the weight filtration of an object $M \in \MNoriLoc(X)$ splits amounts to saying that, for every $a \in \Z$, the short exact sequence
%\begin{equation*}
%	0 \to W_{\leq a-1} M \to M \to W_{\geq a} M \to 0
%\end{equation*}
%splits.
%In fact, since the weight filtration is finite, the conditions needs to be checked only for finitely many values of $a$.
%Since the functor $x^{\dagger} \colon \MNoriLoc(X) \to \MNori(\C)$ preserves the weight filtration up to shifting by $-\dim(X)$ (by \Cref{rem:f^*-pure}), we are led to introduce the following notion: 
\begin{defn}
	For every $a \in \Z$, we define the \textit{weight-$a$-splitting locus} of $M \in \MNoriLoc(X)$ as the subset $\WeightLocus_a(M) \subset X(\C)$ consisting of all points $x \in X(\C)$ such that the short exact sequence
	\begin{equation*}
		0 \to W_{\leq a-1} M \to M \to W_{\geq a} M \to 0
	\end{equation*}
	splits.
\end{defn}

Since the weight filtration on $M$ is finite, we trivially have $\WeightLocus_a(M) = X(\C)$ for $|a| \gg 0$.
The full weight-splitting locus can be recovered as the essentially finite intersection
\begin{equation*}
	\WeightLocus(M) = \bigcap_{a \in \Z} \WeightLocus_a(M).
\end{equation*}
An alternative way to prove the main properties of the full weight-splitting locus (that is, \Cref{thm:WeightLocus} and \Cref{prop:WeightLocus-Gal}) is to establish the same properties for each locus $\WeightLocus_a(M)$.
More generally:

\begin{defn}
	We define the \textit{splitting locus} of a short exact sequence in $\MNoriLoc(X)$ 
	\begin{equation*}
		\mathfrak{S} \colon 0 \to M' \to M \to M'' \to 0
	\end{equation*}
	as the subset $\SplitLocus(\mathfrak{S}) \subset X(\C)$ consisting of all points $x \in X(\C)$ such that the induced short exact sequence in $\MNori(\C)$
	\begin{equation*}
		x^{\dagger} \mathfrak{S} \colon 0 \to x^{\dagger} M' \to x^{\dagger} M \to x^{\dagger} M'' \to 0
	\end{equation*}
	splits.
\end{defn}

We trivially have $\SplitLocus(\mathfrak{S}) = X(\C)$ in case the sequence $\mathfrak{S}$ is already split.
If this is not the case, the splitting locus can be described as follows:
\begin{thm}\label[thm]{thm:ext-splitting_locus}
	For every short exact sequence in $\MNoriLoc(X)$
	\begin{equation*}
		\mathfrak{S} \colon 0 \to M' \to M \to M'' \to 0
	\end{equation*} 
	which is not split, the locus $\SplitLocus(\mathfrak{S})$ is a countable union of subsets of $X(\C)$ of the form $S(\C)$, with $S$ a strict closed algebraic subvariety of $X$.
	Moreover, if $X$ admits a model $X_k$ over an algebraically closed subfield $k \subset \C$ such that $M$ belongs to the essential image of $\MNoriLoc(X_k)$, then each maximal closed subvariety inside $\SplitLocus(\mathfrak{S})$ is defined over $k$ as well.
\end{thm}
\begin{proof}
	As in the proof of \Cref{thm:WeightLocus}, fix a countable algebraically closed subfield $k \subset \C$ such that $X$ admits a model $X_k$ over $k$ and $M$ is isomorphic to the image of an object $M_k \in \MNoriLoc(X_k)$ under the base-change functor $\MNoriLoc(X_k) \to \MNori(X)$.
	Since this functor is fully faithful, with essential image stable under subquotients (by \Cref{prop:MNoriLoc_inv}), the entire short exact sequence $\mathfrak{S}$ arises from a short exact sequence $\mathfrak{S}_k$ in $\MNoriLoc(X_k)$.
	
	With this set-up, the argument for \Cref{thm:WeightLocus} carries through, using directly \Cref{lem:split-bc} instead of \Cref{prop:M_semisimple=Gmot0_reductive}.
	We leave the details to the interested reader. 
\end{proof}

\begin{cor}\label[cor]{cor:split_stalk}
	A short exact sequence $\Sfrak$ in $\MNoriLoc(X)$ splits as soon as the short exact sequence $x^{\dagger} \Sfrak$ in $\MNori(\C)$ splits for all $x \in X(\C)$.
\end{cor}
\begin{proof}
	Indeed, if $\Sfrak$ is not already split then $\SplitLocus(\Sfrak) \neq X(\C)$ (by \Cref{thm:ext-splitting_locus}). 
\end{proof}

For sake of completeness, we also state the stability properties of general splitting loci under Galois conjugation.
To this end, we fix a subfield $k \subset \C$, and we let $\overline{k} \subset \C$ denote its algebraic closure.

\begin{prop}\label[prop]{prop:SplitLocus-Gal}
	Suppose that $X$ admits a model $X_k$ over $k$ such that $\mathfrak{S}$ belongs to the essential image of $\MNoriLoc(X_k)$.
	Then, the $\Gal(\overline{k}/k)$-action on $X$ induced by the Galois action on $X_{\overline{k}}$ stabilizes the locus $\SplitLocus(\mathfrak{S})$.
\end{prop}
\begin{proof}
	As in the proof of \Cref{prop:WeightLocus-Gal}, fix an element $\varphi \in \Gal(\overline{k}/k)$, and let $f \colon X \xrightarrow{\sim} X$ denote the induced scheme automorphism.
	We have to prove that the permutation map on complex points $f \colon X(\C) \xrightarrow{\sim} X(\C)$ stabilizes the subset $\SplitLocus(\mathfrak{S})$.
	
	So fix a point $x \in \SplitLocus(\Sfrak)$, and let us show that $f(x) \in \SplitLocus(\Sfrak)$ as well.
	To this end, it suffices to observe that we have an isomorphisms of short exact sequences in $\MNori(\C)$
	\begin{equation*}
		f(x)^{\dagger} \mathfrak{S} = x^{\dagger} f^* \mathfrak{S} = x^{\dagger} \mathfrak{S},
	\end{equation*}
	where the second passage follows as in the proof of \Cref{prop:MotLocus-Gal}.
	In particular, since $x^{\dagger} \Sfrak$ is assumed to be split, so must be $f(x)^{\dagger} \Sfrak$.
\end{proof}


\begin{thebibliography}{9}
	\bibitem[SGA1]{SGA1} A. Grothendieck,
	\textit{Rev{\^e}tements étales et groupe fondamental}.
	Séminaire de Géometrie Algébrique du Bois-Marie, dirigé par A. Grothendieck, augmenté de deux exposés de Mme M. Raynaud. Lecture Notes in Math. \textbf{224}, Springer-Verlag, Berlin (1964).
	\bibitem[Stacks]{Stacks} The Stacks Project authors,
	\textit{The Stacks Project}.
	\url{https://stacks.math.columbia.edu/}.
	%\bibitem[Anc15]{Anc15} G. Ancona,
	%\textit{Décomposition de motifs abéliens}.
	%Manuscripta Mathematica \textbf{146} (2015), pp. 307--328.
	\bibitem[And92]{And92} Y. André,
	\textit{Mumford--Tate groups of mixed Hodge structures and the theorem of the fixed part}.
	Compos. Math. \textbf{82} (1992), pp. 1--24.
	\bibitem[And96]{And96} Y. André,
	\textit{Pour une théorie inconditionnelle des motifs}.
	Publ. Math. IH{\'E}S \textbf{83} (1996), pp. 5--49.
	%\bibitem[And04]{And04} Y. André,
	%\textit{Une introduction aux motifs (motifs purs, motifs mixtes, périodes)}.
	%Panoramas et Synthèses \textbf{17}, Soc. Math. de France, Paris (2004).
	\bibitem[And21]{And21} Y. André,
	\textit{A note on $1$-motives}.
	Int. Math. Res. Not. (2021), no. 3, pp. 2074--2080.
	\bibitem[Ara13]{Ara13} D. Arapura,
	\textit{An abelian category of motivic sheaves}.
	Adv. Math. \textbf{233} (2013), pp. 135--195.
	\bibitem[Ayo07a]{Ayo07a} J. Ayoub,
	\textit{Les six opérations de Grothendieck et le formalisme des cycles évanescents dans le monde motivique, I}.
	Astérisque (2007), no. 314, x + 466 pp. (2008). 
	\bibitem[Ayo07b]{Ayo07b} J. Ayoub,
	\textit{Les six opérations de Grothendieck et le formalisme des cycles évanescents dans le monde motivique, II}.
	Astérisque (2007), no. 315, vi + 364 pp. (2008).
	\bibitem[Ayo10]{Ayo10} J. Ayoub,
	\textit{Note sur la réalisation de Betti et les opérations de Grothendieck}.
	J. Inst. Math. Jussieu \textbf{9}, no. 2 (2010), pp. 225--263.
	\bibitem[Ayo14a]{Ayo14H1} J. Ayoub,
	\textit{L'algèbre de Hopf et le groups de Galois motiviques d'un corps de caractéristique nulle, I}.
	J. Reine Angew. Math. \textbf{693} (2014), pp. 1--149.
	\bibitem[Ayo14b]{Ayo14H2} J. Ayoub,
	\textit{L'algèbre de Hopf et le groupe de Galois motivique d'un corps de caractéristique nulle, II}.
	J. Reine Angew. Math. \textbf{693} (2014), pp. 151--226.
	\bibitem[Ayo14c]{Ayo14Et} J. Ayoub,
	\textit{La réalisation étale et les opérations de Grothendieck}.
	Ann. Sci. {\'E}c. Norm. Supér. \textbf{47}, no. 1 (2014), pp. 1--141.
	\bibitem[BT25]{BT25} B. Bakker, J. Tsimerman,
	\textit{Functional transcendence of periods and the geometric André--Grothendieck period conjecture}.
	Forum of Math. Sigma \textbf{13}, e97 (2025), pp. 1--24.
	\bibitem[BKU24]{BKU} G. Baldi, B. Klingler, E. Ullmo,
	\textit{On the density of the Hodge locus}.
	Invent. Math. \textbf{235} (2024), pp. 441--487.
	\bibitem[BBU26]{BaBiUr} G. Baldi, G. Binyamini, D. Urbanik,
	\textit{Algebraic Hodge generic points are dense}.
	Preprint (2026), \url{https://arxiv.org/abs/2606.08882}.
	\bibitem[Bei87]{Beilinson} A. Beilinson,
	\textit{On the derived category of perverse sheaves}.
	In: $K$-Theory, Arithmetic and Geometry (Moscow, 1984–1986). Lecture Notes in Math. \textbf{1289}, Springer, Berlin (1987), pp. 27--41.
	\bibitem[BBD82]{BBD82} A. Beilinson, J. Bernstein, P. Deligne,
	\textit{Faisceaux pervers}.
	In: Analysis and topology on singular spaces, I (Luminy, 1981). Astérisque \textbf{100}, Soc. Math. France, Paris (1982), pp. 5--171.
	%\bibitem[Bon14]{Bon14} M.V. Bondarko,
	%\textit{Weights for relative motives: relation with mixed complexes of sheaves}.
	%International Math. Research Notices (2014), no. 17, pp. 4715--4767.
	\bibitem[BPS10]{BPS10} P. Brosnan, G. Pearlstein, C. Schnell,
	\textit{The locus of Hodge classes in an admissible variation of mixed Hodge structures}.
	C. R. Math. Acad. Sci. Paris \textbf{348}, no. 11-12 (2010), pp. 657--660.
	\bibitem[CDK95]{CDK95} E. Cattani, P. Deligne, A. Kaplan,
	\textit{On the locus of Hodge classes}.
	J. Amer. Math. Soc. \textbf{82} (1995), pp. 483--505.
	\bibitem[CG17]{CG17} U. Choudhury, M. Gallauer Alves de Souza,
	\textit{A comparison of motivic Galois groups}.
	Adv. Math. \textbf{313} (2017), pp. 470--536. 
	%\bibitem[CD14]{CisDeg} D.C. Cisinski, F. Déglise,
	%\textit{{\'E}tale motives}.
	%\bibitem[DAE22]{DAE22} M. D'Addezio, H. Esnault,
	%\textit{On the universal extensions in Tannakian categories}.
	%International Math. Research Notices (2022), no. 18, pp. 14008--14033.
	%\bibitem[Del71a]{DelShi} P. Deligne,
	%\textit{Travaux de Shimura}.
	%Séminaire Bourbaki, Exposé 389. Lecture Notes in Math. \textbf{244} (1971), pp. 123-165.
	%\bibitem[Del71b]{DelHodge2} P. Deligne,
	%\textit{Théorie de Hodge II}.
	%Publications Mathématiques de l'IH{\'E}S \textbf{40} (1971), pp. 5--58.
	%\bibitem[Del74]{DelHodge3} P. Deligne,
	%\textit{Théorie de Hodge III}.
	%Publications Mathématiques de l'IH{\'E}S \textbf{44} (1974), pp. 5--77.
	\bibitem[DM82a]{DM82Ab} P. Deligne (notes by J.S. Milne),
	\textit{Hodge cycles on abelian varieties}.
	In: Hodge Cycles, Motives, and Shimura Varieties (Eds. P. Deligne, J.S. Milne, A. Ogus, K. Shih). Lecture Notes in Math. \textbf{900}, Springer, Berlin (1982), pp. 9--100.
	\bibitem[DM82b]{DM82} P. Deligne. J.S. Milne,
	\textit{Tannakian categories}.
	In: Hodge Cycles, Motives, and Shimura Varieties (Eds. P. Deligne, J.S. Milne, A. Ogus, K. Shih). Lecture Notes in Math. \textbf{900}, Springer, Berlin (1982), pp. 101--228.
	\bibitem[Hir64]{Hir64} H. Hironaka,
	\textit{Resolution of singularities of an algebraic variety over a field of characteristic zero. I}.
	Ann. Math. (2) \textbf{79} (1964), pp. 109--123.
	\bibitem[Hub97]{Hub97} A. Huber,
	\textit{Mixed perverse sheaves for schemes over number fields}.
	Compos. Math. \textbf{108} (1997), pp. 107--121.
	\bibitem[HMS17]{HMS17} A. Huber, S. M{\"u}ller-Stach,
	\textit{Periods and Nori motives}.
	Ergeb. Math. Grenzgeb. (3) \textbf{65}, Springer, Cham (2017).
	\bibitem[IM24]{IM24} F. Ivorra, S. Morel,
	\textit{The four operations on perverse motives}.
	J. Eur. Math. Soc. \textbf{26}, no. 11 (2024), pp. 4191--4272.
	\bibitem[Jac26]{Jac25} E. Jacobsen,
	\textit{Malcev completions, Hodge Theory, and Motives}.
	Algebra Number Theory \textbf{20}, no. 1 (2026), pp. 147--193.
	\bibitem[JT25]{JT25} E. Jacobsen, L. Terenzi,
	\textit{A comparison of categories of Nori motivic sheaves}.
	Preprint (2025), \url{https://arxiv.org/abs/2509.21476}.
	\bibitem[Kli17]{KliConj} B. Klingler,
	\textit{Hodge loci and atypical intersections: conjectures}.
	To appear in: Motives and Complex Multiplication.
	\url{https://arxiv.org/abs/1711.09387}.
	\bibitem[KO21]{KO21} B. Klingler, A. Otwinowska,
	\textit{On the closure of the Hodge locus of positive period dimension}.
	Invent. Math. \textbf{225} (3), pp. 857--883 (2021).
	\bibitem[KOU23]{KOU23} B. Klingler, A. Otwinowska, D. Urbanik,
	\textit{On the fields of definition of Hodge loci}.
	Ann. Sci. {\'E}c. Norm. Supér. \textbf{56} (2023), pp. 1299--1312.
	%\bibitem[Kot92]{Kot92} R.E. Kottwitz,
	%\textit{Points on some Shimura varieties over finite fields}.
	%Journal of the AMS \textbf{5}, no. 2 (1992), pp. 373--444.
	\bibitem[Kre22a]{KreutzAbs} T. Kreutz,
	\textit{Absolutely special subvarieties and absolute Hodge cycles}.
	Preprint (2022), \url{https://arxiv.org/abs/2111.00216}.
	\bibitem[Kre22b]{KreutzEll} T. Kreutz,
	\textit{$\ell$-Galois special subvarieties and the Mumford--Tate conjecture}.
	Preprint (2022), \url{https://arxiv.org/abs/2111.01126}.
	\bibitem[Mor25]{Mor24} S. Morel,
	\textit{Mixed $\ell$-adic complexes for schemes over number fields}.
	Doc. Math. \textbf{30}, no. 1 (2025), pp. 105--181.
	%\bibitem[PS08]{PS08} C. Peters, J. Steenbrick,
	%\textit{Mixed Hodge structures}.
	%Ergeb. Math. Grenzgeb. \textbf{52}, Springer, Berlin, Heidelberg (2008).
	%\bibitem[Sai90]{Sai90} M. Saito,
	%\textit{Mixed Hodge modules}.
	%Publications of the RIMS \textbf{26}, no. 2 (1990), pp. 221--333.
	\bibitem[SS16]{SaitoSchnell} M. Saito, C. Schnell,
	\textit{Fields of definition of Hodge loci}.
	In: Recent advances in Hodge Theory (Eds. M. Kerr, G. Pearlstein), LMS Lecture Note Series \textbf{427} (2016), pp. 275--291.
	%\bibitem[Sim90]{Simpson} C. Simpson,
	%\textit{Transcendental aspects of the Riemann--Hilbert correspondence}.
	%Illinois Journal of Math. \textbf{34}, no. 2 (1990), pp. 368--391.
	\bibitem[Ter24]{Ter24UAF} L. Terenzi,
	\textit{On the functoriality of universal abelian factorizations}.
	Preprint (2024), \url{https://arxiv.org/abs/2401.13583}.
	\bibitem[Ter26]{Ter24Nori} L. Terenzi,
	\textit{Tensor structure on perverse Nori motives}.
	Ann. $K$-Theory \textbf{11}, no. 1 (2026), pp. 47--170.
	\bibitem[Tub25a]{Tub25} S. Tubach,
	\textit{On the Nori and Hodge realisations of Voevodsky motives}.
	Compos. Math. \textbf{161}, no.9 (2025), pp. 2155--2201.
	\bibitem[Tub25b]{TubThesis} S. Tubach,
	\textit{Motifs de Nori relatifs et foncteurs de réalisation}.
	Ph.D. thesis, ENS de Lyon.
	\url{https://theses.hal.science/tel-05261833}.
	\bibitem[Tub26]{Tub25Artin} S. Tubach,
	\textit{Artin motives in relative Nori and Voevodsky motives}.
	J. Algebra \textbf{690} (2026), pp. 400--424.
	\bibitem[Voi07]{Voisin} C. Voisin,
	\textit{Hodge loci and absolute Hodge classes}.
	Compos. Math. \textbf{143}, no. 4 (2007), pp. 945--958.
	\bibitem[Wei79]{Weil} A. Weil,
	\textit{Abelian varieties and the Hodge ring}.
	In: André Weil: Collected Papers, Vol. III, Springer Verlag (1979), pp. 421--429.
\end{thebibliography}
\end{document}